\documentclass[11pt,a4paper]{article}

\usepackage{authblk}
\usepackage{url}
\usepackage{amssymb,latexsym,makeidx,amsmath,
amsfonts,amsthm}
\usepackage{mathpartir}
        \usepackage{upgreek}

\usepackage{amsmath}
\usepackage{amssymb}

\usepackage{hyperref}
\usepackage{url}
\usepackage{enumerate}

\usepackage{amssymb,latexsym,makeidx,amsfonts}
\usepackage{mathpartir}
\usepackage{upgreek}

\usepackage{xcolor}

\newcommand\commented[3]{{$\langle$ {\color{#1}  #2 }{#3} $\rangle$}}

\newcommand\hide[1]{\commented{gray}{Hidden:}{#1}}
\renewcommand\hide[1]\empty

\newcommand\extended[1]{\commented{blue}{Extra:}{#1}}
\renewcommand\extended[1]\empty

\newcommand\DS[1]{\commented{blue}{DS:}{\it   #1}}
\renewcommand\DS[1]\empty

\newcommand\ISH[1]{\commented{magenta}{ISH:}{{\it   #1}}}
\renewcommand\ISH[1]\empty

\newcommand\newISH[1]{\commented{red}{ISH, Oct. 2018 :}{{\it   #1}}}
\renewcommand\newISH[1]\empty

\newcommand\novISH[1]{\commented{blue}{I, Nov. :}{{\it   #1}}}
\renewcommand\novISH[1]\empty

\newcommand\marISH[1]{\commented{blue}{I, March:}{{\it   #1}}}
\renewcommand\marISH[1]\empty

\newcommand\novDS[1]{\commented{red}{D, Nov.  :}{{\it   #1}}}
\renewcommand\novDS[1]\empty

\newcommand\now[1]{\commented{red}{IS, today: }{#1}}
\renewcommand\now[1]{\empty}

\newcommand\DSnow[1]{\commented{violet}{DS, today: }{#1}}
\renewcommand\DSnow[1]{\empty}

\def\filtR{\clS}
\def\filtR{\overline{\clR}}

\newcommand\Pair[2]{#2[#1]}
\def\KL{\Pair{K}{L}}

\def\Th{\mathit{Th}}
\def\Ths{\classts{T}}
\def\clS{\classts{S}}
\def\Mod{\mathit{Mod}}

\newtheorem{theorem}{Theorem}
\newtheorem{lemma}[theorem]{Lemma}
\newtheorem{proposition}[theorem]{Proposition}

\newtheorem{corollary}[theorem]{Corollary}

\newtheorem*{theorem*}{Theorem}
\newtheorem*{lemma*}{Lemma}
\newtheorem*{proposition*}{Proposition}
\newtheorem*{prop*}{Proposition}
\newtheorem*{corollary*}{Corollary}
\newtheorem*{cor*}{Corollary}

\theoremstyle{remark}
\newtheorem{remark}{Remark}
\newtheorem*{remark*}{Remark}

\newtheorem*{examp*}{Example}

\newtheorem*{question*}{Question}

\def\tiff{\; \text{iff}\;}

\theoremstyle{definition}
\newtheorem{definition}{Definition}
\newtheorem*{definition*}{Definition}
\newtheorem{example}{Example}

\newtheorem*{example*}{Example}
\newtheorem*{examples*}{Examples}

\def\kpp{{(\kappa)}}
\def\conc{{\smallfrown}}
\newcommand\vlf[1]{ \| #1 \|}
\def\vl{{\vlf{\!\cdot\!}}}
\newcommand\logicts[1]{\mathrm{#1}}
\newcommand\LogicNamets[1]{\logicts{#1}}

\newcommand\lS[1]{\LogicNamets{S#1}}

\def\ML{\logicts{MF}}
\def\PV{\logicts{PV}}
\def\MTh{\logicts{MTh}}
\def\MLog{\logicts{MLog}}
\def\vf{\varphi}
\def\ZFC{\logicts{ZFC}}
\def\PA{\logicts{PA}}
\def\isom{\simeq}

\def\MF{\mathrm{MF}}

\newcommand\rulets[1]{\mathop{#1}}

\def\Sub{\rulets{Sub}}

\def\quot{\geq}

\def\mvf{\upvarphi}
\def\mpv{\mathsf p}
\def\mpsi{\uppsi}

\def\submod{\sqsubseteq}
\def\supmod{\sqsupseteq}

\newcommand\structurets[1]{\mathfrak #1}
\def\stA{\structurets{A}}
\def\stB{\structurets{B}}

\def\frF{\structurets{F}}

\def\frT{\structurets{T}}

\def\mM{\structurets{M}}

\newcommand\classts[1]{\mathcal{#1}}
\def\clC{\classts{C}}
\def\clA{\classts{A}}

\def\clD{\classts{D}}
\def\clP{\classts{P}}

\def\Rmax{\clR_\Ths^\diamond}
\def\Rsimmax{\clR_\sim^\diamond}
\def\clR{\classts{R}}
\def\clV{\classts{V}}

\def\L{L}

\def\mo{\models}

\def\Imp{\Rightarrow}
\def\imp{\rightarrow}
\def\Di{\Diamond}

\def\con{\wedge}
\def\mo{\vDash}
\def\EE{\exists}
\def\AA{\forall}

\def\Iff{\Leftrightarrow}
\def\lra{\leftrightarrow}

\newcommand\dom{\mathrm{dom}}
\newcommand\ran{\mathrm{ran}}

\title{On modal logics of model-theoretic relations\footnote{The work on this paper was supported by the Russian Science Foundation under grant 16-11-10252 and carried out at Steklov Mathematical Institute of Russian Academy of Sciences.}}
\author[a,b]{Denis I.~Saveliev}
\author[a,b]{Ilya B.~Shapirovsky}
\affil[a]{{\small Steklov Mathematical Institute of Russian Academy of Sciences}}
\affil[b]{{\small 
Institute for Information Transmission Problems of Russian Academy of Sciences
}}

\begin{document}

\maketitle

\sloppy




\begin{abstract}
Given a class $\clC$ of models, a~binary
relation~$\clR$ between models, and a~model-theoretic
language~$L$, we consider the modal logic and the
modal algebra
of the theory
of $\clC$ in~$L$ where
the modal operator is interpreted via~$\clR$.
We discuss how modal theories of $\clC$ and $\clR$ depend on the model-theoretic language, their Kripke completeness,
and expressibility of the modality inside $L$.
We calculate such theories for the submodel and the quotient relations.
We prove a downward L\"owenheim--Skolem theorem
for first-order language expanded with
the modal operator for the extension relation between models.

\medskip
\noindent
Keywords: \textit{modal logic,
modal algebra, robust modal theory,
logic of submodels,
logic of quotients,
logic of forcing,
provability logic,
model-theoretic logic
}
\end{abstract}

\section*{Introduction}
We consider modal systems in which the modal operator is interpreted via a binary relation on a class of models.
Many instances of such systems can be found in the literature.
During the last years, modal logics of various relations between models of set theory have been studied,
see, e.g., \cite{Hamkins03,HamkLowe,BlockLowe2015,ForcStruct2015,innerMods2016}.
A well established area in provability logic deals with modal axiomatizations of relations between models of arithmetic (and between arithmetic theories), see, e.g., \cite{Shavrukov1988,Berarducci1990,Ignatiev93,VisserBigFirst,Visser14, Henk2015,HamkinsArithmeticPotentialism2018}.
In another extensively studied area modalities are interpreted by relations
between  Kripke and temporal models,
see, e.g., \cite{Veltman96, agostino_hollenberg_2000,TemporalSubst2013}
or the monograph~\cite{vanBenthem2014}.
In \cite{BarwiseVBenthem1999}, the consequence along
an abstract relation between models is studied, which
is closely related to our consideration.


\smallskip

Let  $f$ be a unary operation on sentences of
a~model-theoretic language $L$, and $T$ a~set of sentences
of~$L$ (e.g., the set of theorems in a~given calculus, or
the set of sentences valid in a~given class of models).
Using the propositional modal language, one can consider
the following ``fragment'' of~$T$: variables are evaluated
by sentences of $L$, and $f$ interprets the modal operator;
the {\em modal theory of $f$ on~$T$}, or just the
{\em $f$-fragment of $T$}, is defined as the set of those
modal formulas which are in $T$ under every valuation.
A well-known example of this approach is a~complete modal
axiomatization of formal provability in Peano arithmetic
given by Solovay~\cite{Solovay1976}.
Another important example is the theorem by Hamkins and
L\"{o}we axiomatizing the modal logic of forcing
(introduced earlier by Hamkins in~\cite{Hamkins03})
where the modal operator expresses satisfiability in
forcing extensions~\cite{HamkLowe}. Both these modal
systems have good semantic and algorithmic properties;
in particular, they have the finite model property, are
finitely axiomatizable, and hence decidable.

These examples inspire the following observation.
Let $\clC$ be an arbitrary class of models of the same
signature, $T=\Th^L(\clC)$ the theory of $\clC$ in
a~model-theoretic language~$L$, and $\clR$ a~binary relation
on~$\clC$. Assuming that the satisfiability in $\clR$-images
of models in $\clC$ can be expressed by an operation $f$
on sentences of~$L$, i.e., for every sentence $\vf$ of~$L$,
and every $\stA\in\clC$,
\begin{quote}
$\stA\mo f(\varphi)$
\;(``$\varphi$~is possible at~$\stA$'')\;
iff $\stB\mo\vf$ for some $\mathfrak B$ with
$\stA\,\clR\,\stB$,
\end{quote}
we can define the {\em modal theory of $\clR$ in~$L$}
as the $f$-fragment of~$T$. In the
general frame semantics, this modal theory is
characterized by an enormous structure
$
(\mathcal C,\mathcal R,\mathcal C_\vf:
\vf\text{ is a~sentence of~}L)
$
where $\clC_\vf$~is the class of models in~$\clC$
validating~$\vf$.
We can also define the {\em modal Lindenbaum algebra
of $\Th^L(\clC)$ and~$\clR$}, i.e., the Boolean algebra
of sentences of $L$ modulo the equivalence on~$\clC$,
endowed with the modal operator induced by~$f$.%
\footnote{Modal algebras of theories are broadly used
in provability logic, where they are called
{\em Magari} (or {\em diagonizable}) {\em algebras};
they are known to keep a~lot of information about
 theories containing arithmetic
(see, e.g.,~\cite{ShavrukovPhD93}).
}
In Section~\ref{sec:defs}, we provide formal definitions
and basic semantic tools for such modal theories.
In particular, the algebra of $\Th^L(\clC)$ and $\clR$
can be represented as the modal algebra of a~general frame
consisting of complete theories $\{\Th^L(\stA):\stA\in\clC\}$.
We use this in Section~\ref{sec:sub}, where we calculate
modal logics of the submodel relation~$\supmod$; to express
the satisfiability in submodels, we use second-order language.

In Section \ref{sec:upanddown},%
\footnote{A~significant part of Sections
\ref{sec:upanddown} and \ref{sec:rob} was motivated
by reviewers' questions on an earlier version of this paper.}
we discuss the situation when the $\clR$-satisfiability
is not expressible in a~language~$K$ (for example,
this situation is typical when $K$ is first-order).
%
In this case the modal algebra of $\Th^K(\clC)$ can be
defined as the subalgebra of the modal algebra of
$\Th^L(\clC)$ generated by the sentences of~$K$, for
any $L$ stronger than~$K$ where the $\clR$-satisfiability
on~$\clC$ is expressible: the resulting modal algebra
(and hence, its modal logic) does not depend on the way
how we extend the language~$K$. Under a~natural assumption
on $\clC$ and~$\clR$, such an $L$ can always by
constructed, and hence, {\em the} modal algebra of
$\Th^K(\clC)$ and $\clR$ is well-defined for arbitrary~$K$
(however, the resulting modal logic is not necessarily
a~``fragment'' of the theory $\Th^K(\clC)$ anymore).
Then we consider the finitary first-order language
%
expanded with the modal operator for the extension
relation~$\submod$ and prove a~version of the downward
L\"owenheim--Skolem theorem for this language.

In general, modal theories of $\clR$ depend on the
model-theoretic language we consider. We say that
a modal theory of $\clR$ is {\em robust} iff making
the language stronger does not alter this theory
(intuitively, the robust theory can be considered as
the~``true'' modal logic of the model-theoretic
relation~$\mathcal R$). We discuss this notion in
Section~\ref{sec:rob}. In Theorem~\ref{thm: robust logicGOOD},
we show that under a~certain assumption on $\clC$ and~$\clR$,
the robust logic is Kripke complete. Then we use this
theorem to describe robust theories of the quotient and
the submodel relations on certain natural classes.

A~preliminary report on some results in Sections \ref{sec:defs}
and \ref{sec:sub} can be found in~\cite{SavelievShapirovsky2016}.

\section{Preliminaries}
To simplify reading, as a~rule we denote the syntax of
our object languages differently: we use inclined letters
if they are related to model theory ($x,y,\ldots$ for
individual variables, $\varphi,\psi,\ldots$ for formulas),
and upright letters if they are related to modal logic
($\mathsf p,\mathsf q,\ldots$ for propositional variables,
$\upvarphi,\uppsi,\ldots$ for formulas).

\paragraph{Model-theoretic languages}

The languages we use for model theory are model-theoretic
languages in sense of~\cite{BarwiseFeferman} (where they
are called ``model-theoretic logics'').
We are
\hide{
For the precise
definition, we refer the reader to~\cite{BarwiseFeferman}
\marISH{A kilo-page book, maybe  we need a more precise reference? Also,
this sounds like a repeat.
 Moreover, not every ``logic'' contains
$L_{\omega,\omega}$. }
}%
pointing out here only two of their features, which will
be essential for the further investigations: if $L$~is
a~model-theoretic language under our consideration, then:
\begin{itemize}
\item[(i)]
Satisfiability in~$L$ is preserved under isomorphisms;
\item[(ii)]
$L$~includes $L_{\omega,\omega}$, the standard first-order
language with usual finitary connectives and quantifiers.
\end{itemize}

Also, we assume that $L$~is a~set unless otherwise specified.
For a~model $\stA$, $\Th^L(\stA)$, or just $\Th(\stA)$,
denotes its theory in $L$, i.e., the set of all sentences
of $L$ holding in~$\stA$; likewise for a~class of models.

\paragraph{Modal logic}\label{subs:basicML}
By a~{\em modal logic} we shall mean a~normal propositional
unimodal logic (see, e.g., \cite{BRV-ML} or~\cite{CZ}).
Modal formulas are built from a~countable set~$\PV$
of propositional variables
$\mathsf{p},\mathsf{q},\ldots$\,,
the Boolean connectives,
and the modal connectives $\Di$,~$\Box$.
We assume that $\bot$, $\imp$, $\Di$ are basic,
and others are abbreviations; in particular,
$\Box$~abbreviates $\neg\,\Di\,\neg$\,.
Let $\ML$~denote the set of modal formulas.
A~set of modal formulas is called a~{\em normal
modal logic} iff it contains all classical
tautologies, two following modal axioms:
$$\neg\,\Di\bot
\;\;\text{and}\;\;
\Di(\mathsf p\vee\mathsf q)\to\Di\mathsf p\vee\Di\mathsf q,
$$
and is closed under three following rules of inference:
$$
\text{Modus Ponens}\;
\frac{\mvf,\;\mvf\to\mpsi}{\mpsi},
\;\;
\text{Substitution}\;
\frac{\mvf(\mathsf p)}{\mvf(\mpsi)},
\;\;
\text{Monotonicity}\;
\frac{\mvf\to\mpsi}{\Di\mvf\to\Di\mpsi}.
$$

A~{\em general frame}~$\frF$ is a~triple $(W,R,\clV)$
where $W$~is a~nonempty set, $R\subseteq W\times W$,
and $\clV$~is a~subalgebra of the powerset algebra
$\clP(W)$ closed under~$R^{-1}$ (i.e., such that
$R^{-1}(A)=\{x\in W:\exists y\in A\,\,x R y\}$
is in~$\clV$ whenever $A$~is in~$\clV$).
By a~{\em Kripke frame\/} $(W,R)$ we mean $(W,R,\clP(W))$.
We shall often use the term {\em frame} for
a~general frame (and never for a~Kripke frame).
A {\em modal algebra} is a~Boolean algebra endowed
with a~unary operation that distributes w.r.t.~finite
disjunctions. The {\em modal algebra of a~frame}
$(W,R,\clV)$ is $(\clV,R^{-1})$.

A~{\em Kripke model}~$\mathfrak M$ on a~frame
$\mathfrak F=(W,R,\clV)$ is a~pair $(\mathfrak F,\theta)$
where $\theta:\PV\to\clV$ is a~{\em valuation}.
The {\em truth relation\/} $\mM,x\mo \mvf$
(``$\mvf$~is true at $x$ in~$\mM$'')
is defined in the standard way (see, e.g.,~\cite{CZ});
in particular, $\mM,x\mo\Di\mvf$ iff $\mM,y\mo\mvf$
for some $y$ with $xRy$.
A~formula $\mvf\in\MF$ is {\em true in a~model~$\mM$}
iff it is true at every $x$ in~$\mM$,
and $\mvf$~is {\em valid in a~frame~$\frF$}
iff it is true in every model on~$\mathfrak F$.
A modal formula $\mvf$
is {\em valid in a~modal algebra}~$\stA$
iff $\mvf=\top$ holds in~$\stA$. In particular, it follows
that if $\stA$ is the modal algebra of a frame $\frF$, then
$\mvf$ is valid in~$\stA$ iff $\mvf$~is valid in~$\frF$.
The set $\MLog(\frF)$ of all valid in~$\frF$
formulas is called the ({\em modal})
{\em logic of the frame~$\frF$}; likewise for algebras.
It is well-known (see, e.g., \cite{BRV-ML} or~\cite{CZ})
that a~set of formulas is a~consistent normal logic iff
it is the logic of a~general frame iff it is the logic
of a~non-trivial modal algebra.

In our paper we also consider modal logics of
general frames $(\mathcal C,\mathcal R,\mathcal V)$
with a~proper class~$\mathcal C$;
see Remarks \ref{rem:claclaclass1}
and~\ref{rem:claclaclassGeneral} in the next section.

$\lS{4}$~is the smallest logic containing the formulas
$\Di\Di\mathsf p\imp\Di\mathsf p$
and $\mathsf p\imp\Di\mathsf p$,
$\lS{4.2}$~is $\lS{4}+\Di\Box\mathsf p\;{\rightarrow}\;\Box\Di\mathsf p$,
and $\lS{4.2.1}$ is
$
\lS{4}+
\Box\Di\mathsf p\;{\leftrightarrow}\;\Di\Box\mathsf p
$,
where $\Lambda+\mvf$ is the smallest logic
that includes $\Lambda\cup\{\mvf\}$.

\section{Modal theories of model-theoretic relations}\label{sec:defs}

\paragraph{Definitions}
Fix a~signature~$\Omega$ and a~language $L=L(\Omega)$
based on~$\Omega$. Let $L_s$~be the set of all sentences
of~$L$.

Consider a~unary operation~$f$ on sentences of~$L$.
A~({\em propositional}) {\em valuation in~$L$} is
a~map from $\PV$ to~$L_s$. A~valuation~$\vl$ extends
to the set~$\MF$ of modal formulas as follows:
\begin{align*}
\vlf{\bot}
&=
\bot,
\\
\vlf{\mvf\to\mpsi}
&=
\vlf{\mvf}\to\vlf{\mpsi},
\\
\vlf{\Di\mvf}
&=
f(\vlf{\mvf}).
\end{align*}
Given a~set of sentences $T\subseteq L_s$,
we define
the {\em modal theory of $f$ on $T$}  $\mathrm{MTh}(T,f)$
as the set of modal formulas $\mvf$ such that
  for every valuation $\vl$ in $L$,  $\vlf{\mvf}$ is in $T$.
Thus, modal formulas can be viewed as axiom schemas (for sentences) and we can think of the modal theory of $f$ on $T$ as a fragment of $T$.

\begin{example}
Let $L$~be the usual first-order finitary language
$L_{\omega,\omega}$.

1.
Let $\Omega$~be the signature of arithmetic
and
$f(\varphi)$ express the consistency of a sentence $\varphi$ in Peano arithmetic~$\logicts{PA}$. By the well-known Solovay's results~\cite{Solovay1976},
the modal theory of $f$ on $\PA$ is the G\"odel--L\"ob
logic~$\logicts{GL}$, and
the modal theory of $f$ on the true arithmetic~$\logicts{TA}$, the
set of all sentences that are true in the standard model of arithmetic, is the (quasi-normal but not normal) Solovay logic~$\logicts{S}$.
A~survey on provability logics can be found
 in~\cite{ArtemovBekl04Prov}.

2.
Let $\Omega$~be the signature of set theory
(i.e., $\Omega=\{\in\}$) and $f(\varphi)$~express that
$\varphi$ holds in a~generic extension. By Hamkins and
L\"owe~\cite{HamkLowe}, the modal theory of $f$
on $\logicts{ZFC}$ is~$\logicts{S4.2}$.
\end{example}

In our work, we are  interested in the case when
$T$~is the theory of some class $\clC$ of models
of~$\Omega$.
An easy observation shows that in this case the
modal theory is closed under the rules of Modus Ponens
and Substitution. Notice that it is not necessarily
closed under Monotonicity, as the instance of the
Solovay logic~$\logicts{S}$ shows.%
\footnote{Let us also mention the logic of {\em pure
provability} introduced by Buss in~\cite{buss1990},
which is not closed under the rule of substitution.}
However, as we shall see, the modal theory is a~normal
logic whenever $f$~expresses the satisfiability in
images of a~binary relation between models.

\begin{definition}
Let $\clR$~be a~binary relation on~$\clC$ (or,
perhaps, on a~larger class $\clC'\supseteq\clC$).
The {\em $\clR$-satisfiability on $\clC$ is
expressible in}~$L$ iff there exists $f:L_s\to L_s$
such that for every sentence $\vf\in L_s$
and every model $\stA\in\clC$,
\begin{equation*}\label{eq:r-sat}
\stA\vDash f(\vf)
\;\Iff\;
\exists\,\stB\in\clC\;
(\stA\,\clR\,\stB\;\&\;\stB\vDash\vf).
\end{equation*}
\end{definition}

Some examples of such expressibility and
non-expressibility will be given below.

The following is straightforward:

\begin{proposition}
If $f$ and $f'$ both express the
$\mathcal R$-satisfiability on~$\clC$,
and $T=\Th(\clC)$, then
$\mathrm{MTh}(T,f)=\mathrm{MTh}(T,f').$
\end{proposition}

Thus, in this case, $\mathrm{MTh}(T,f)$
does not depend on the choice of~$f$.

\begin{definition}\label{def:exprMTH}
Let the $\mathcal R$-satisfiability be expressible
in $L$ on a~class~$\mathcal C$ of models.
The {\em modal theory of $(\mathcal C,\mathcal R)$
in~$L$}, denoted by $\MTh^L(\mathcal C,\mathcal R)$,
is the set of modal formulas $\mvf$ such that
$\clC\mo\vlf{\mvf}$ for every propositional valuation
in~$L$.
\end{definition}

In this section, we just write $\mathrm{MTh}(\clC,\clR)$
assuming $L$~is fixed.

Given a propositional valuation $\vl$ in $L$, consider
$\mM=(\clC,\clR,\theta)$, the enormous
``Kripke model'' with
$\theta(\mathsf p)=\{\stA\in\clC:\stA\mo\vlf{\mathsf p}\}$.
By a straightforward induction on a~modal formula~$\mvf$,
for every $\stA$ in $\clC$
$$
\stA\mo\vlf{\mvf}
\;\Iff\;
\mM,\stA\mo\mvf,
$$
where $\mo$ on the left-hand and on the right-hand side
of the equivalence denotes the satisfiability relation
in model theory and the truth relation of a~modal formula in a Kripke model, respectively.
It follows that
\begin{equation*}
\clC\mo\vlf{\mvf}
\;\Iff\;
\mM,\stA\mo\mvf\text{ for all }\stA\in\clC.
\end{equation*}
Let $\Mod(\psi)$ be the class of models of $\psi\in\L_s$,
and $\clC_\psi=\Mod(\psi)\cap \clC$.
Then  validity of $\vlf{\mvf}$ in $\clC$ for all propositional valuations in $L$
can be considered as validity of the modal formula $\mvf$ in an enormous
``general frame of models'' $(\clC,\clR,\clV)$ where
$\clV$ ``consists'' of $\clC_\psi$ with $\psi\in L_s$.

\begin{remark}\label{rem:claclaclass1}
The collection $\mathcal V$ looks like  a~``class of classes''.
In fact, however, $\mathcal V$~is a~usual class defined
by a~formula of set theory with two parameters:
if $\mathcal C$, $L_s$ are defined by formulas
$\varPhi$, $\varPhi'$, respectively, then
what we understand by~$\mathcal V$ is the class of pairs
$(\mathfrak A,\psi)$ defined by a~formula~$\varUpsilon$
that is constructed from $\varPhi$, $\varPhi'$
and expresses the satisfaction relation between models
in $\mathcal C$ and sentences in~$L_s$. Thus subclasses
$\mathcal C_\psi$ of $\clC$ playing the role of
``elements'' of~$\mathcal V$ are in fact defined by
$\varUpsilon$ with a~fixed second argument~$\psi$.
This allows us to formally extend the definition of
validity in a~(set) frame to the case of the ``frame of models''
 $(\clC,\clR,\clV)$.
\end{remark}
It follows that the modal logic $\MLog(\mathcal C,\mathcal R,\mathcal V)$, i.e.,
 the set of all modal formulas that are valid in $(\mathcal C,\mathcal R,\mathcal V)$,  coincides with the modal theory
of  $(\mathcal C,\mathcal R)$ in~$L$.
Namely, we have

\begin{theorem}\label{thm:big-as-general}
If the $\mathcal R$-satisfiability on~$\mathcal C$
is expressible in~$L$, then
\begin{equation}\label{eq:big-as-general}
\mathrm{MTh}(\mathcal C,\mathcal R)=
\MLog(\mathcal C,\mathcal R,\mathcal V).
\end{equation}
\end{theorem}

Consequently, $\MTh(\mathcal C,\mathcal R)$ is a~normal logic.

\hide{
\begin{remark}\label{rem:claclaclass1-old}
The family  $\clV$ looks  like  a ``class of classes''.
However, $\clV$~can be  considered as a~usual class defined by a formula of set theory with two parameters:  if
$\mathcal C=\{x:\varPhi(x)\}$,
$L_s=\{y:\varPsi(y)\}$, and $\varUpsilon(x,y)$ defines the $L$-satisfaction relation,  then what we understand by~$\mathcal V$ in our setting
is the  class
$\bigl\{
(x,y):
\varPhi(x)\;\&\;
\varPsi(y)\;\&\;
\varUpsilon(x,y))
\bigr\}.
$
\end{remark}
}

\begin{proposition}\label{prop:extendingexpr}
Assume that $f$  expresses the $\clR$-satisfiability
on $\clC$ in $L$.
\begin{enumerate}
\item
If $\clD\subseteq \clC$ is $\clR$-upward closed
(i.e.,
$\stA\in \clD\;\&\;\stA\,\clR\,\stB\;\&\;\stB\in\clC$
implies $\stB\in\clD$),
then the $\clR$-satisfiability on~$\clD$
is expressible in $L$ by~$f$.
\item
If $\psi$ is a~sentence of~$L$, then the
$\clR$-satisfiability on $\clC\cap\Mod(\psi)$ is
expressible in $L$ by $g:\vf\mapsto f(\vf\wedge\psi)$.
\end{enumerate}
\end{proposition}

Let $\clR^*(\stA)$ denote $\bigcup_{n<\omega} \clR^n(\stA)$,
the least $\clR$-closed $\clD\subseteq\clC$ containing~$\stA$.
From (\ref{eq:big-as-general})
and the generated subframe construction
(see, e.g., \cite[Section 8.5]{CZ}),
we obtain

\begin{corollary}\label{cor:gener}
If the $\clR$-satisfiability on $\clC$
is expressible in~$L$, then
\begin{center}
$
\MTh(\clC,\clR)\,=\,
\bigcap\limits_{\mathfrak A\in\clC}
\MTh(\clR^*(\mathfrak A),\clR).
$
\end{center}
\end{corollary}

\paragraph{Frames of theories}
Assume that the $\clR$-satisfiability is
expressible on $\clC$ in $L$ by some $f:L_s\to L_s$.

We have observed in Theorem~\ref{thm:big-as-general}
that $\MTh(\mathcal C,\mathcal R)$ can be viewed
as the modal logic of a~general frame of models.
Now we provide other semantic characterizations
of $\MTh(\mathcal C,\mathcal R)$.

Put $\Ths=\{\Th(\stA): \stA\in\clC\}$. Note that $\Ths$ is a set since $L$ is assumed to be a set.
For theories $T_1,T_2 \in \Ths$, let
\begin{eqnarray*}
T_1\,\clR_{\Ths}\,T_2
&\tiff&
\EE\,\stA_1,\stA_2\in\clC\,\;
(\stA_1\mo T_1\;\&\;\stA_2\mo T_2\;\&\;\stA_1\clR\,\stA_2),
\\
T_1\,\Rmax\,T_2
&\tiff&
\AA\vf\in T_2\;\,f(\vf)\in T_1.
\end{eqnarray*}
Observe that $\Rmax$~does not depend on the choice
of~$f$, and $\clR_{\Ths}\subseteq\Rmax$.

For $\vf\in L_s$, put
$\Ths_\vf=\{T\in\Ths:\vf\in T\}$.
Clearly,
$\Ths_{\vf\con\psi}=\Ths_{\vf}\cap\Ths_{\psi}$
and $\Ths_{\neg\,\vf}=\Ths\setminus\Ths_{\vf}$.
Consider an arbitrary binary relation $\filtR$ such that
$$\clR_{\Ths}\subseteq\filtR\subseteq\Rmax.$$
It follows from the definitions that
$\filtR^{-1}(\Ths_\vf)=\Ths_{f(\vf)}$.
Therefore, $$\clA=\{\Ths_{\vf}:\vf\in L_s\}$$ is
a~subalgebra of the powerset algebra $\clP(\Ths)$
closed under $\filtR^{-1}$, and $(\Ths,\filtR,\clA)$
is a~general frame. We call $(\Ths,\clR_{\Ths},\clA)$
and $(\Ths,\Rmax,\clA)$
the {\em minimal} and the {\em maximal frames
of theories for $\clC,\clR$, and~$L$}.

\begin{remark}
The relation $\filtR$ can be viewed as
a~{\em filtration  of $\clR$}
(see, e.g., \cite[Section 4]{Goldblatt92}).
The frame $(\Ths,\Rmax,\clA)$ is known as
the {\em refinement} of $(\clC,\clR,\mathcal V)$
(cf.~\cite[Chapter 8]{CZ}).
\end{remark}

For sentences $\vf,\psi$, put $\vf\approx\psi$
iff $\clC\mo\vf\lra \psi$.
Let $L/{\approx}$ be the Lindenbaum algebra of
the theory $\Th(\clC)$, i.e., the set~$L_s$ of sentences
of $L$ modulo~$\approx$ with operations induced by
Boolean connectives. We have $f(\vf)\approx f(\psi)$
whenever $\vf\approx\psi$, hence, $f$~induces
the operation~$f_\approx$ on $L/{\approx}$.
It is easy to see that $(L/{\approx},f_\approx)$
is a~modal algebra. It is called the
{\em modal (Lindenbaum) algebra of the language~$L$
on $(\clC,\clR)$}.

The above arguments yield

\begin{theorem}\label{thm:minmax-ths}
If $\clR_{\Ths}\subseteq \filtR\subseteq\Rmax$\,,
then the modal algebras $(L/{\approx},f_\approx)$
and $(\clA,\filtR^{-1})$ are isomorphic.
\end{theorem}
\begin{proof}
The isomorphism takes the $\approx$-class
of a~sentence~$\vf$ to~$\Ths_{\vf}$.
\end{proof}

\begin{theorem}\label{thm:maintoolnew}
Assume that the $\clR$-satisfiability on $\clC$
is expressible in~$L$. Let
$\clR_{\Ths}\subseteq \filtR \subseteq \Rmax$.
Then
$$
\MTh(\clC,\clR)=
\MLog(\Ths,\filtR,\clA)=
\MLog(L/{\approx},f_\approx).
$$
\end{theorem}

\begin{proof}
Every valuation $\vl$  in $L$ can be viewed as
a~valuation~$\theta$ in the frame $\frF=(\Ths,\filtR,\clA)$,
and vice versa. By induction on~$\mvf$,
for every $\stA$ in~$\clC$ we have\,
$
\stA\mo\vlf{\mvf}\tiff
(\frF,\theta){,}\Th(\stA)\mo\mvf.
$
Thus $\MTh(\clC,\clR)=\MLog(\Ths,\filtR,\clA)$.

The second equality immediately follows from
Theorem~\ref{thm:minmax-ths}.
\end{proof}

\begin{corollary}\label{prop: finite frame}
If $\Ths$ is finite, then
$\mathrm{MTh}(\mathcal C,\mathcal R)$ is the logic
of the Kripke frame $(\Ths,\clR_{\Ths})$.
\end{corollary}

\begin{proof}
In this case $\clA=\clP(\Ths)$.
\end{proof}

The family $\Ths$ of complete theories can be viewed
as the quotient of $\clC$ by the $L$-equivalence~$\equiv$
where $\stA\equiv \stB$ iff $\Th(\stA)=\Th(\stB)$.
Theorem~\ref{thm:maintoolnew} can be generalized
for the case of any equivalence~$\sim$ on~$\clC$
finer than~$\equiv$\,; in particular, it holds
for the isomorphism equivalence $\isom$ on models, or for
the equivalence in a~stronger model-theoretic language.
Namely, we let:
\begin{eqnarray*}
\,[\stA]_\sim\clR_\sim[\stB]_\sim
&\,\Iff\,&
\exists\,\mathfrak A'\sim\mathfrak A\;
\exists\,\mathfrak B'\sim\mathfrak B\;\,
\mathfrak A'\,\mathcal R\,\mathfrak B',
\\
\,[\stA]_\sim\Rsimmax[\stB]_\sim
&\Iff&
\forall\,\varphi\in L_s\;\,
(\mathfrak B\vDash\varphi\,\Rightarrow\,
\mathfrak A\vDash f(\varphi));
\end{eqnarray*}
the algebra of valuations is defined as
the collection~$\clV_\sim$ ``consisting''
of classes~$\clC_\psi/{\sim}$ for $\psi\in L_s$.
(Again, since $[\mathfrak A]_\sim$ are in general
proper classes, $\clV_\sim$~looks like
a~``class of classes of classes'' but actually is
nothing but a~three-parameter formula; cf.~Remarks
\ref{rem:claclaclass1} and~\ref{rem:claclaclassGeneral}.)
As in the proof of~Theorem \ref{thm:maintoolnew},
for $\clR_\sim\subseteq\filtR\subseteq \Rsimmax$
one can obtain that
\begin{equation}\label{eq:quot}
\MTh(\clC,\clR)=\MLog(\clC/{\sim},\filtR,\clV_\sim).
\end{equation}

To the best of our knowledge,
Theorem~\ref{thm:maintoolnew}
(or its analogue~(\ref{eq:quot})) has never been
formulated  explicitly before,
although similar constructions related to
frames of arithmetic theories were considered earlier
(V.\,Yu.~Shavrukov, an unpublished note, 2013;
\cite[Remark 3]{Henk2015}).

\smallskip
We conclude this section with a continuation
of Remark~\ref{rem:claclaclass1}.

\begin{remark}\label{rem:claclaclassGeneral}
\marISH{new}
It is possible to give a~general definition of frames
that are classes and their logics (in our metatheory
which is assumed to be $\ZFC$ where ``classes'' are
shorthands for formulas). We outline the idea and
postpone details for a~further paper.
Below capital Greek letters $\varPhi,\varPsi,\ldots$
denote formulas of the metatheory.

Assume that $\mathcal C$, $\mathcal L$, $\mathcal R$
are (arbitrary) classes defined by formulas
$\varPhi$, $\varPhi'$, $\varPsi$, respectively:
$\mathcal C=\{x:\varPhi(x)\}$,
$\mathcal L=\{y:\varPhi'(y)\}$, and
$\mathcal R=\{(x,v):\varPsi(x,v)\}$.
Let us say that
a~class~$\mathcal V$ forms a~\emph{class modal algebra}
(admissible for $\mathcal C$, $\mathcal L$,~$\mathcal R$)
iff it is defined by a~formula~$\varUpsilon$  that
fulfills the conditions expressing
that classes
$\mathcal C_y=\{x:\varUpsilon(x,y)\}$ indexed by $y$
in~$\mathcal L$ play the role of ``elements''
of~$\mathcal V$. Namely, $\varUpsilon$ implies that
all $\clC_y$ are subclasses of~$\mathcal C$
and their collection is closed under Boolean operations
and the modal operator given by~$\mathcal R$;
e.g., the latter is expressed as follows:
$$
\forall y\,\exists z\,\forall x\;
(\varUpsilon(x,z)\:\&\:\varPhi'(y)
\;\Leftrightarrow\;
\varPhi(x)\:\&\:\exists v\:
(\varPsi(x,v)\:\&\:\varUpsilon(v,y))).
$$

Further, let us say that $\theta$~is a~\emph{valuation}
of propositional variables in~$\mathcal V$ iff $\theta$~is
a~(set) function with $\dom(\theta)=\PV$ and
$\ran(\theta)\subseteq\mathcal L$, and that
$\mathfrak F=(\mathcal C,\mathcal R,\mathcal V)$
is a~\emph{class general frame} and
$\mathfrak M=(\mathcal C,\mathcal R,\mathcal V,\theta)$
a~\emph{class Kripke model} on~$\mathfrak F$. To define
the \emph{truth} of modal formulas~$\upvarphi$ at
a~point $x\in\mathcal C$ in the model~$\mathfrak M$,
denoted by $\mathfrak M,x\vDash\upvarphi,$
we first extend $\theta$ to a~suitable~$\bar\theta$
with $\dom(\bar\theta)=\MF$ and
$\ran(\theta)\subseteq\mathcal L$. It can be shown
that such an extension exists and is unique up to
the equivalence~$\sim$ on~$\mathcal L$ defined
by letting $y\sim z$ iff
$
\forall x\,
(\varUpsilon(x,y)\Leftrightarrow\varUpsilon(x,z)).
$
Then we let
\begin{align*}
\mathfrak M,x\vDash\upvarphi
\;\Leftrightarrow\;
\varUpsilon(x,\bar\theta(\upvarphi)).
\end{align*}
This notion of truth has the expected properties;
e.g., we have
\begin{align*}
\mathfrak M,x\vDash\Diamond\upvarphi
&\;\Leftrightarrow\;
\exists v\:(\varPsi(x,v)\:\&\:
\mathfrak M,v\vDash\upvarphi).
\end{align*}
A~formula $\upvarphi\in\MF$ is
\emph{true} in~$\mathfrak M$ iff
$
\forall x\,(\varPhi(x)\Rightarrow
\mM,x
\vDash\upvarphi),
$
and \emph{valid} in~$\frF$ iff it is true
in all models~$\mathfrak M$ on~$\frF$.
The \emph{modal logic $\MLog(\frF)$ of the class
frame~$\frF$} consists of those $\upvarphi\in\MF$ that
are valid in~$\frF$. It can be verified that this logic
is a~set defined by a~$\ZFC$-formula constructed from
the formulas $\varPhi,\varPhi',\varPsi,\varUpsilon$,
and it is normal.
%

By using formulas with additional parameters, one
can imitate higher order class algebras and their
modal logics.
\hide{
E.g., for $\varUpsilon(x,y,z)$,
classes $\mathcal C_{y,z}=\{x:\varUpsilon(x,y,z)\}$
play the role of ``elements'' of the ``superclass''
$\mathcal C_{y}=\{x:\exists z\,\varUpsilon(x,y,z)\}$,
which in turn plays the role of an ``element''
of~$\mathcal V$.\marISH{Consolidate}
}
\end{remark}

\section{Logics of submodels}\label{sec:sub}
In this part, we apply Theorem \ref{thm:maintoolnew} to calculate
the modal theory of the {\it submodel relation\/} on the
class of all models of a given signature.

\paragraph{Expressing the satisfiability}\label{subsec:relativ}
Given models $\mathfrak A$ and
$\mathfrak B$ of a~signature~$\Omega$, let
$\mathfrak A\supmod\mathfrak B$ mean
``$\mathfrak A$~contains $\mathfrak B$ as a~submodel''.
As the initial step, we find a~model-theoretic
language for expressing the
$\sqsupseteq$-satisfiability.
Observe first that
  first-order languages are generally too weak for this.

\begin{proposition}\label{prop:submod-non-exprr}
If $\Omega$~contains a~predicate symbol of arity~$\ge2$,
then the $\sqsupseteq$-satisfiability is not
expressible in~$L_{\omega,\omega}$ and moreover,
in the infinitary language~$L_{\infty,\omega}$.
\end{proposition}

\begin{proof}
We can suppose w.l.g.~that $\Omega$ contains a~binary
predicate symbol~$<$ (otherwise mimic it by a~predicate
symbol of a~bigger arity by fixing other arguments).
Toward a~contradiction, assume that some~$f$ mapping
the class of $L_{\infty,\omega}$-sentences into itself
expresses the $\sqsupseteq$-satisfiability. Let
$\varphi$~be an obvious $L_{\omega,\omega}$-sentence
saying that there exists a~$<$-minimal element, and
let $\psi$~be the sentence $\neg\,f(\neg\,\varphi)$.
Then $\psi$~says that each submodel has a~$<$-minimal
element (thus whenever $\Omega$~has no functional
symbols then $\psi$~says that $<$~is well-founded).
Let $\kappa$~be such that $\psi\in L_{\kappa,\omega}$.
It follows from Karp's theorem (see, e.g.,
\cite[Theorem~14.29]{rosenstein1982linear}) that there
are models $\mathfrak A_0$ and $\mathfrak B_0$ of
$\Omega_0=\{<\}$ such that $\mathfrak A_0$~is isomorphic
to an ordinal while $\mathfrak B_0$~is not, and
$\mathfrak A_0\equiv_{L_{\kappa,\omega}}\!\mathfrak B_0$.
Add a~$<$-last element to each of the models
$\mathfrak A_0$ and $\mathfrak B_0$ and check
that the resulting models $\mathfrak A_1$ and
$\mathfrak B_1$ remain $L_{\kappa,\omega}$-equivalent
(e.g., by using \cite[Lemma~14.24]{rosenstein1982linear}).

Expand $\mathfrak A_1$ and $\mathfrak B_1$ to models
$\mathfrak A$ and $\mathfrak B$ of~$\Omega$, respectively,
by interpreting each predicate symbol other than~$<$
by the empty set, each functional symbol of positive
arity by the projection onto the first argument, and
each constant symbol by the $<$-last element of the model.
It is easy to see that in both $\mathfrak A$ and
$\mathfrak B$ any formula of~$\Omega$ is equivalent
to a~formula of~$\Omega_0$; so we still have
$\mathfrak A\equiv_{L_{\kappa,\omega}}\!\mathfrak B$.
On the other hand, in both models every subset forms
a~submodel whenever it contains the $<$-last element of
the whole model, whence it easily follows that
$\mathfrak A\vDash\psi$ and $\mathfrak B\vDash\neg\,\psi$.
A~contradiction.
\end{proof}

\hide{

\begin{proposition}\label{prop:submod-non-exprr}
If $\Omega$~consists of a~binary predicate symbol,
then the $\sqsupseteq$-satisfiability is not
expressible in~$L_{\omega,\omega}$ and moreover,
in the infinitary language~$L_{\infty,\omega}$.
\end{proposition}

\begin{proof}
Toward a~contradiction, assume that some $f$
expresses the $\sqsupseteq$-satisfiability.
Let $\varphi$~be an obvious
$L_{\omega,\omega}$-sentence
saying that there exists a~$\leq$-minimal element
(where $\leq$ is the~binary predicate symbol in~$\Omega$).
Then $\neg\,f(\neg\,\varphi)$ expresses
the existence of a~$\leq$-minimal element in each non-empty
subset, i.e., the well-foundedness of~$\leq$.
The latter is not expressible even in~$L_{\infty,\omega}$
(see, e.g., \cite[Theorem 3.2.20]{BarwiseFeferman}).
\end{proof}
}

However, second-order language suffices to express
the $\sqsupseteq$-satisfiability by the relativization
argument (see, e.g., \cite[p.~242]{QuantHand14-2007}).
Given a~second-order formula~$\varphi$ and a~unary
predicate variable $U$ that does not occur in~$\varphi$,
let $\varphi^U$~be the relativization of~$\varphi$ to~$U$
defined in the standard way; in particular, if $P$ and $F$
are second-order predicate functional variables of arity~$n$,
then
\begin{itemize}
\item[]
$(\exists P\,\varphi)^U$ is
$
\exists P\,
\bigl(\forall x_0\ldots\forall x_{n-1}\,
\bigl(P(x_0,\ldots,x_{n-1})\to
\bigwedge_{i<n}U(x_i)\bigr)\wedge\varphi^U\bigr),
$
\item[]
$(\exists F\,\varphi)^U$ is
$
\exists F\,
\bigl(\forall x_0\ldots\forall x_{n-1}\,
\bigl(\bigwedge_{i<n}U(x_i)\to
U(F(x_0,\ldots,x_{n-1}))\bigr)
\wedge\varphi^U\bigr).
$
\end{itemize}
Let $\psi(U)$ be the formula expressing that $U$~is
a~submodel, i.e., saying that the interpretation of~$U$
is non-empty and is closed under interpretations of
functional symbols in~$\Omega$. Then the map
$\varphi\mapsto\exists U(\psi(U)\wedge\varphi^U)$
expresses the $\sqsupseteq$-satisfiability on the class
of all models of~$\Omega$. In view of
Proposition~\ref{prop:extendingexpr},
we obtain:

\begin{proposition}\label{prop:submod-quot-exprr}
Let $\kappa=|\{F\in\Omega:F$~is a~functional symbol$\}|.$
\begin{itemize}
\item[1.]
Let $\mathcal C$~be a~class of models of~$\Omega$ closed
under submodels. Then the $\sqsupseteq$-satisfiability is
expressible on $\mathcal C$ in $L^{2}_{\lambda,\omega}$
whenever $\lambda>\kappa$ and $\lambda\ge\omega$.
\item[2.]
Let $T$~be a~set of sentences of $L^{2}_{\mu,\omega}$
of~$\Omega$. Then the $\sqsupseteq$-satisfiability is
expressible on the class $Mod(T)$ in $L^{2}_{\lambda,\omega}$
whenever $\lambda>\max(\kappa,|T|)$ and $\lambda\geq\mu$.
\end{itemize}
\end{proposition}

\begin{remark}
These results on expressibility can be refined in several
directions. In particular, the first statement of
Proposition~\ref{prop:submod-quot-exprr}
remains true for monadic language~$L^{2}_{\lambda,\omega}$;
the assumption $\lambda>\kappa$ is necessary;
for details and further results, see~\cite{Saveliev2019}.
\hide{
For instance, it holds for
monadic second order language (maybe, with infinitary
conjunctions for the case of infinite $\Omega$ and~$T$).
}
\end{remark}

%

\paragraph{Axiomatization}

Henceforth in this section we assume that $L$~expresses
the $\supmod$-satisfiability on the class of models
under consideration.

The next easy result is soundness for modal theories
of the submodel relation.

\begin{theorem}\label{thm:sqsup-soundness}
Let $\clC$ be a~class of $\Omega$-models closed
under submodels. Then $\MTh^L(\clC,\supmod)$ is
a~normal modal logic including~$\mathrm{S4}$.
If moreover, $\Omega$~contains a~constant symbol,
then $\MTh^L(\clC,\supmod)$ includes $\mathrm{S4.2.1}$.
\end{theorem}
\begin{proof}
Let $\stA\in\clC$, $\vl$ a valuation in $L$. Trivially, $\stA\mo\vlf{\Di\Di \mathsf{p}\imp \Di \mathsf{p}}$
and $\stA\mo\vlf{\mathsf{p}\imp \Di \mathsf{p}}$.
If $\Omega$ contains a constant symbol, consider the submodel $\stB$ of $\stA$ generated by constants. It is straightforward that in this case
${\stA\mo \vlf{\Box\Di \mathsf{p}}} \tiff \stB\mo  \vlf{\Di \mathsf{p}} \tiff \stB\mo  \vlf{\Box  \mathsf{p}} \tiff
\stA \mo \vlf{\Di \Box \mathsf{p}}.$
\end{proof}

We are going to prove completeness.
Let $Q_n$ be the lexicographic  product
of $(n^{<n},\subseteq)$ (an $n$-ramified tree
of height~$n$) and $(n,n\!\times\!n)$ (a~cluster
of size~$n$). Thus for $s,t\in n^{<n}$
and $i,j\in n$, in $Q_n$ we have
$$
(s,i)\le(t,j)
\;\Iff\;
s\subseteq t,
$$
so $Q_n$~is a~pre-tree which is $n$-ramified, has
height~$n$ and clusters of size~$n$ at each point.
Let also $Q_n'$~be the ordered sum of $Q_n$ and
a~reflexive singleton, thus $Q'_n$~adds to~$Q_n$
an extra top element. The following fact is standard (see,~e.g.,~\cite[p. 563]{CZ}).

\begin{proposition}\label{prop:preetreeS4}
Let $\mvf$ be a modal formula.
If $\mvf\notin\lS{4}$, then $\mvf$ is not valid in $Q_n$ for some $n>0$.
If $\mvf\notin\lS{4.2.1}$, then $\mvf$ is not valid in $Q'_n$ for some $n>0$.
\end{proposition}

Let $\equiv$ be the $L$-equivalence.
For a model  $\mathfrak A$, let
$\Sub(\stA)$ be the set of all its submodels,
$\Sub(\stA)_\equiv$ abbreviate $\Sub(\stA)/{\equiv}$.

\begin{theorem}\label{thm:models-An}
Let $\Omega$~have a~functional symbol of arity~$\geq 2$.
For every positive $n<\omega$, there exists
a~model~$\mathfrak A_n$ of~$\Omega$ such that
\[
\bigl(\Sub(\mathfrak A_n)_\equiv,\sqsupseteq_\equiv\bigr)
\text{ is isomorphic to }
\left\{
\begin{array}{ll}
Q_n & \text{ if $\Omega$ has no constant symbols},
\\
Q'_n& \text{ otherwise}.
\end{array}
\right.
\]
\end{theorem}

\begin{proof}
First, suppose that $\Omega$~has no constant symbols.

Without loss of generality we may assume that
$\Omega$~has a~binary functional symbol;
we write  $\,\cdot\;$ for it.

Fix $n\geq 1$.
Let $X_n=n^{<n}\times \omega$. We define the model
$\mathfrak{A}_n=(X_n,\cdot\,,\dots)$ of~$\Omega$
as follows. Let $E$~be any injective map from
 $n^{<n}$
into~$\omega$. For $s,t\in n^{<n}$ and $i,j\in\omega$, we
put
\begin{align*}
(s,i)\cdot(t,j)=
\left\{
\begin{array}{ll}
(s,i+1)
&\text{if }s=t,\,i=j,
\\
(s^{\conc}({i\bmod n}),j)
&\text{if }s=t,\,j<i,\,|s|<n-1,
\\
(s,j+E(s))
&\text{if }s=t,\,j=i+1,\,i\equiv_n\!0,
\\
(\inf\{s,t\},\inf\{i,j\})
&\textrm{otherwise},
\end{array}
\right.
\end{align*}
where ${}^{\conc}$ denotes the concatenation,
$|s|$~the length of~$s$, $\bmod$~the remainder,
and $\equiv_n$~the congruence modulo~$n$.
For other operations~$F$ in $\mathfrak{A}_n$ we put
$F((s,i),\ldots)=(s,i)$ and take the relations
in $\mathfrak{A}_n$ to be empty.

For $(s,i)\in X_n$ let
$
X_n(s,i)=
\{(t,j)\in X_n:(s,i)\preceq(t,j)\}
$,
where we let $(s,i)\preceq(t,j)$
iff $s\subseteq t$ and $i\le j$.

\begin{lemma}\label{lem:Aui}
An $X\subseteq X_n$ is the universe
of a~submodel of $\mathfrak{A}_n$ iff
$X=X_n(s,i)$ for some $(s,i)\in X_n$.
\end{lemma}

\begin{proof}
Straightforward from the definition of $\mathfrak{A}_n$.
\end{proof}

Let $\mathfrak{A}_n(s,i)$ be the submodel of $\mathfrak{A}_n$ with the universe $X_n(s,i)$.

\begin{lemma}\label{lem:isom}
Let $(s,i),(s,j)\in X_n$. If $i\equiv_n\!j$, then the
models $\mathfrak{A}_n(s,i)$ and $\mathfrak{A}_n(s,j)$
are isomorphic.
\end{lemma}

\begin{proof}
The map $(t,l)\mapsto(t,l+n)$ is an isomorphism between
$\mathfrak{A}_n(s,i)$ and $\mathfrak{A}_n(s,i+n)$.
\end{proof}
We define $p^0(x)$ as $x$, and $p^{k+1}(x)$ as $(p^{k}(x))\cdot (p^{k}(x))$. Then  for $k<\omega$
$$
\mathfrak{A}_n\vDash(t,j)=p^k(s,i)
\quad \text{iff}\quad s=t \text{ and } j=i+k.
$$
For $s\in n^{<n}$, let $\varphi_s(x)$ be
the following one-parameter formula:
$$
x\cdot p(x)=p^{E(s)+1}(x).
$$

\begin{lemma}\label{lem:x1}
Let $\mathfrak{A}$ be a~submodel of $\mathfrak{A}_n$
and $(t,i)$ an element of~$\mathfrak{A}$. Then
$\mathfrak{A}\vDash\varphi_s(t,i)$
iff $s=t$ and $i\equiv_n\!0$.
\end{lemma}

\begin{proof}
We have $p(t,i)=(t,i+1)$.
By the definition, $(t,i)\cdot(t,i+1)=(t,i+E(t)+1)$ if
$i\equiv_n\!0$, and $(t,i)\cdot(t,i+1)=(t,i)$ otherwise.
\end{proof}

\begin{lemma}\label{lem:x2}
Let $\mathfrak{A}$ be a~submodel of $\mathfrak{A}_n$. For every $s\in n^{<n}$, we have:
 {$\mathfrak{A}\vDash\exists x\,\varphi_s(x)$} iff
 $(s,i)$ is in $\mathfrak{A}$ for some
$i\in\omega$.
\end{lemma}

\begin{proof}
The `only if' part is immediate from Lemma~\ref{lem:x1}. For the `if' part we use Lemmas \ref{lem:x1} and \ref{lem:Aui}.
\end{proof}

For $S\subseteq n^{<n}$, let
$\chi_S$~be the sentence
$
\bigwedge_{s\in S}\exists x\,\varphi_s(x)
\;\wedge\,
\bigwedge_{s\notin S}\neg\,\exists x\,\varphi_s(x),
$
and let $\chi_{\ge s}$ be $\chi_S$ for
$S=\{t\in n^{<n}:s\subseteq t\}$.

\begin{lemma}\label{lem:x3}
Let $\mathfrak{A}$ be a~submodel of $\mathfrak{A}_n$.
Then $\mathfrak{A}\vDash\chi_{\ge s}$ iff
$\mathfrak{A}=\mathfrak{A}_n(s,i)$ for some $i\in\omega$.
\end{lemma}
\begin{proof}
Follows from Lemmas \ref{lem:x2} and \ref{lem:Aui}.
\end{proof}

Let $\psi(x)$ be the formula $\neg\,\exists y\,(x=p(y)).$ Then
$$\mathfrak A_n(s,i)\vDash\psi(t,j) \text{ iff } j=i.$$

For $s\in n^{<n}$ and $k<n$, let
$\chi_{s,k}$ be the following sentence:
$$
\exists x\,
\bigl(\varphi_s(p^{n-k}(x))\wedge\psi(x)\bigr)
\wedge
\chi_{\ge s}.
$$

\begin{lemma}\label{lem:character}
For every submodel $\mathfrak{A}$ of $\mathfrak{A}_n$ and every $k<n$, we have
$\mathfrak{A}\vDash\chi_{s,k}$  iff
$\mathfrak{A}=\mathfrak{A}_n(s,i)$ for some $i$ such that $i\equiv_n k$.
\end{lemma}

\begin{proof}
Follows from Lemmas \ref{lem:x3} and \ref{lem:x1}.
\end{proof}

\begin{lemma}\label{lem:comp-important}
Let $\mathfrak{A},\mathfrak{B}$ be submodels of
$\mathfrak{A}_n$.  The following
are equivalent:
\begin{enumerate}
\setlength\itemsep{0em}
\item[(i)]
$\mathfrak{A}$ and $\mathfrak{B}$ are isomorphic,
\item[(ii)]
$\mathfrak{A}$ and $\mathfrak{B}$ are $L$-equivalent,
\item[(iii)]
$\mathfrak{A}$ and $\mathfrak{B}$ are elementarily equivalent, i.e., $L_{\omega,\omega}$-equivalent,
\item[(iv)]
$\mathfrak{A}=\mathfrak{A}_n(s,i)$ and $\mathfrak{B}=\mathfrak{A}_n(s,j)$
for some $s\in n^{<n}$ and $i,j<\omega$ such that $i\equiv_n j$.
\end{enumerate}
\end{lemma}

\begin{proof}
The implications {(i)$\,\Imp\,$(ii)} and {(ii)$\,\Imp\,$(iii)} are immediate from our basic assumptions on model-theoretic languages.
The crucial step {(iii)$\,\Imp\,$(iv)} follows from Lemmas \ref{lem:Aui} and  \ref{lem:character}.
Finally, {(iv)$\,\Imp\,$(i)} holds by Lemma \ref{lem:isom}.
\end{proof}

From Lemma \ref{lem:Aui} we conclude that $(\Sub(\mathfrak{A}_n),\sqsupseteq)$
is isomorphic to $(X_n,\preceq)$. Now it follows that
$(\Sub(\mathfrak{A}_n)_\equiv,\sqsupseteq_\equiv)$
is isomorphic to~$Q_n$, as required.

\medskip
For the case when $\Omega$~has constant symbols,
we add a~new element~$c$ to $X_n$ and define the model
$\mathfrak{A}'_n$ on the set $X_n\cup\{c\}$. We extend
the above defined operation~$\cdot$ by letting
$c\cdot x=x\cdot c=c$ for all~$x$; all constant symbols
in~$\Omega$ are interpreted by~$c$.
The same arguments as above prove that
$(\Sub(\mathfrak{A}'_n)_\equiv,\sqsupseteq_\equiv)$
is isomorphic to~$Q'_n$.

This completes the proof of Theorem \ref{thm:models-An}.
\end{proof}

Now the completeness result follows:

\begin{theorem}\label{thm:sqsup-completeness}
Let $\Omega$~contain a~functional symbol of arity~$\ge 2$,
and let $\mathcal C$~be the class of all models of~$\Omega$.
Then
$$
\mathrm{MTh}^L(\mathcal C,\sqsupseteq)=
\left\{
\begin{array}{ll}
\mathrm{S4}
&\text{ if $\Omega$~has no constant symbols},
\\
\mathrm{S4.2.1}
&\text{ otherwise}.
\end{array}
\right.
$$
\end{theorem}

\begin{proof}
We have soundness by Theorem \ref{thm:sqsup-soundness}.
On the other hand, if $\stA\in \clC$, then
$\mathrm{MTh}^L(\mathcal C,\sqsupseteq)$ is contained
in the logic of the Kripke frame
$(\Sub(\stA)_\equiv,\sqsupseteq_\equiv)$
by Corollaries \ref{cor:gener} and~\ref{prop: finite frame}.
Now completeness follows from
Theorem \ref{thm:models-An}.
\end{proof}

Notice that the binary operation used in the proof
of Theorem~\ref{thm:models-An} is not associative.
Modal axiomatizations of $\sqsupseteq$ in second-order
language on semigroups, monoids, and groups are open
questions.

Recall that Theorem~\ref{thm:sqsup-completeness}
was formulated under the assumption that $L$~expresses
the $\sqsupseteq$-satisfiability.
In the next section we discuss how to define
modal theories without such requirements. The logics
calculated in Theorem~\ref{thm:sqsup-completeness}
do not depend on~$L$, while in general, modal theories
depend on a~chosen model-theoretic language; we shall
discuss this in Section~\ref{sec:rob}.

\section{Inexpressible modalities}\label{sec:upanddown}

In this section, we discuss the situation when
the $\clR$-satisfiability is not expressible
in a~model-theoretic language.

\marISH{Reread}
\paragraph{Definition}

The following question was raised by one of
the reviewers on an earlier version of the paper:
what is the modal theory of $(\clC,\clR)$ if
the $\clR$-satisfiability on $\clC$ is not
expressible in a given language? In particular,
what is the modal theory of the relation~$\supmod$
in the first-order case?
Another natural question concerns the definition
of modal theories in the case of~$\submod$.

\hide{
Here we use the notation with superscripts
like $\mathrm{MTh}^{L}(\clC,\clR)$ to denote
the modal theory of  $(\mathcal C,\mathcal R)$ in~$L$.\marISH{hide}
}

\smallskip
Assume that $L$ is a~language stronger than~$K$,
and the $\clR$-satisfiability on $\clC$ is expressible
in~$L$. In this case, $\MTh^{K}(\clC,\clR)$ can be
naturally defined as the logic of the subalgebra of
the modal algebra of $L$ on $(\clC,\clR)$ generated
by the sentences of~$K$. Let us provide details.

First, we define an interim notion $\MTh^{\KL}(\clC,\clR)$.
By Theorem~\ref{thm:maintoolnew}, $\MTh^{L}(\clC,\clR)$
is the logic of
the Lindenbaum modal algebra
$\clA(L)=(L/{\approx},f^L_\approx)$.
Let $\clA(\KL)$ be the subalgebra of $\clA(L)$ generated by
the sentences of~$K$
(more formally,  $\clA(\KL)$ is generated by the set $\{[\vf]^L_\approx: \vf\in K_s\}$, where
$[\vf]^L_\approx=\{\psi\in L_s: \clC\mo \vf\lra\psi\}$).
Define $\MTh^{\KL}(\clC,\clR)$ as the modal logic
of $\clA(\KL)$. (Formally, we have defined
$\MTh^{\KL}(\clC,\clR)$ for the case $K\subseteq L$,
but the same construction works for the case when
$K$~is weaker than $L$ in~$\clC$, i.e., if every
$K$-definable subclass of $\clC$ is definable in~$L$.)

It is immediate that
$\MTh^{\KL}(\clC,\clR)$ is a~normal logic, which includes
$\MTh^L(\clC,\clR)$. Another simple (but important)
observation is that $\MTh^{\KL}(\clC,\clR)$ does not
depend on the choice of~$L$. Indeed, if $M$~is another
language stronger than~$K$, which expresses the
$\clR$-satisfiability on~$\clC$, then the algebras
$\clA(\KL)$ and $\clA(\Pair{K}{M})$ are isomorphic:
their elements can be thought as classes of models
obtained from definable in $K$ subclasses of~$\clC$
via Boolean operations and~$\clR^{-1}$. Thus we have

\begin{proposition}\label{prop:unspeak}
If languages $L$ and $M$ express the
$\clR$-satisfiability on~$\clC$ and are stronger
than~$K$, then:
\begin{enumerate}
\item
the algebras $\clA(\KL)$ and $\clA(\Pair{K}{M})$
are isomorphic, and so
\item
$\MTh^{\KL}(\clC,\clR)=\MTh^{\Pair{K}{M}}(\clC,\clR)$.
\end{enumerate}
\end{proposition}

\begin{definition}[tentative]\label{def:secondMain}
If the $\clR$-satisfiability on $\clC$ is expressible
in some language $L$ stronger than~$K$, put
$\MTh^K(\clC,\clR)=\MTh^{\KL}(\clC,\clR)$.
\end{definition}

Trivially,
$\MTh^{\Pair{L}{L}}(\clC,\clR)=\MTh^L(\clC,\clR)$
whenever $L$ expresses the $\clR$-satisfiability
on~$\clC$. Thus this definition generalizes
Definition~\ref{def:exprMTH}. Note that, however,
the modal logic $\MTh^K(\clC,\clR)$ is not necessarily
a~``fragment'' of the theory $\Th^K(\clC)$ anymore.

\marISH{Reread}
Once we do not require the $\clR$-satisfiability
to be expressible in the language,
we can give an ``external''
definition of   modal theory and   modal algebra
for $\clC$, $\clR$, and~$K$.
Let $K_\Di$ be the set
of modal  {\em $K$-sentences}, which are built
from sentences of $K$ using Boolean connectives
and~$\Di$; namely, they are expressions of form
$\mvf(\psi_1,\ldots,\psi_n)$ where
$\mvf(\mathsf p_1,\ldots,\mathsf p_n)$ is a~modal
formula and $\psi_i$ are sentences of~$K$.
Such languages are regularly used in the context
of modal logics of relations between models of
arithmetic or set theory.%
\footnote{In~\cite{HamkinsArithmeticPotentialism2018},
the language~$K_\Di$ is called
{\em partial potentialist}.}
Given $\clC$ and $\clR$, the $K_\Di$-satisfaction
relation is defined in the straightforward way;
in particular, $\stA\mo\Di\mvf(\psi_1,\ldots,\psi_k)$
iff $\stB\mo\mvf(\psi_1,\ldots,\psi_k)$
for some $\stB\in\clC$ with $\stA\,\clR\,\stB$.
Assume that $\clR$ and $\clC$ satisfy
the following natural condition:
for all $\stA,\stA',\stB$ in~$\clC$,
\begin{equation}\label{eq:isom-bisim}
\stA\isom\stA'\;\&\;\stA\:\clR\:\stB
\;\Rightarrow\;
\EE\stB'\in\clC\;
(\stB'\isom\stB\;\&\;\stA'\,\clR\,\stB')
\end{equation}
(this condition says that the isomorphism equivalence $\isom$ is
a~{\em bisimulation} w.r.t.~$\clR$ on~$\clC$).
In this case the $K_\Di$-satisfaction relation
is preserved under isomorphisms, and so $K_\Di$
satisfies our basic assumptions on model-theoretic
languages.

The operation on sentences of $K_{\Di}$ that takes $\vf$
to $\Di\vf$ expresses the $\clR$-satisfiability
on $\clC$ in $K_{\Di}$. Hence,  we can apply
the constructions described in Section~\ref{sec:defs}
to~$K_\Di$\,. In particular, $\vf\approx \psi$ implies
$\Di\vf\approx\Di\psi$ (recall that $\vf\approx\psi$
means $\clC\mo\vf\lra\psi$), so the connective~$\Di$
induces the operation $\Di_\approx$ on~$K_\Di/{\approx}$,
and $\clA(K_\Di)=(K_\Di/{\approx},\Di_\approx)$ is
a~modal algebra. It is immediate that $\clA(K_\Di)=\clA(K_\Di[K])$,
and hence $$\MTh^{K}(\clC,\clR)=\MTh^{K_\Di}(\clC,\clR)$$
 by Definition \ref{def:secondMain}.
Therefore, assuming that $\clC$ and $\clR$
satisfy~(\ref{eq:isom-bisim}), we obtain
the following generalization of our previous
definitions for arbitrary~$K$:

\begin{definition}\label{def:external-sat}
Let $\clC$ and $\clR$ satisfy~(\ref{eq:isom-bisim}).
The {\em modal (Lindenbaum) algebra of a~language~$K$
on $(\clC,\clR)$} is the modal algebra
$({K_\Di/{\approx}},\Di_\approx)$.
The {\em modal theory $\MTh^K(\clC,\clR)$ of $(\clC,\clR)$
in~$K$} is the modal logic of this algebra.
\end{definition}

\hide{\marISH{WRONG!}
Notice that if the $\clR$-satisfiability on $\clC$ is expressible in some $L$, then (\ref{eq:isom-bisim}) holds. Thus,
(\ref{eq:isom-bisim})
 is a necessary and sufficient condition
for $\clR$ to be expressible on $\clC$ in some $L$.
}

Let us emphasize that natural classes of models with
relations between them always satisfy~(\ref{eq:isom-bisim});
in particular, so are the instances discussed in our paper.
One might say that $\clR$ is a~{\em model theoretic
relation on}~$\clC$ iff property~(\ref{eq:isom-bisim}) holds.
E.g., the submodel relation on the class
of models of a given signature is model-theoretic;
hence, Theorem~\ref{thm:sqsup-completeness} remains
true for arbitrary model-theoretic language.
\hide{
Now we can reformulate Theorem~\ref{thm:sqsup-completeness} for arbitrary $K$.
\begin{corollary}\label{cor:first-robust-sub}
Let $K$ be any language  (in particular, the first-order
language $L_{\omega,\omega}$), let $\Omega$ contain
a~functional symbol of arity~$\geq 2$, and let
$\mathcal C$~be the class of all models of~$\Omega$.
Then $\MTh^K(\clC,\supmod)=\lS{4}$ if $\Omega$ has no
constant symbols, and $\MTh^K(\clC,\supmod)=\lS{4.2.1}$
otherwise.
\end{corollary}
}
\marISH{trivial rudiment from old versions. Why do we need it at all, I forgotten?}

\hide{
\begin{proof}
We have $\MTh^K(\clC,\supmod)=\MTh^{K_\Di}(\clC,\supmod)$ by Definition \ref{def:external-sat} (and, formally,  Theorem \ref{thm:maintoolnew}).
Now the proof follows from Theorem~\ref{thm:sqsup-completeness}.\marISH{new}
\end{proof}
}

\paragraph{Downward L\"owenheim--Skolem theorem}

Our next result provides a~version of the downward
L\"owenheim--Skolem theorem for first-order language
enriched with the modal operator for the extension
relation between models.

For a~cardinal $\kappa$, put
$\clC_{\leq \kappa}=\{\stA\in\clC:|\stA|\leq\kappa\}$.

\begin{theorem}\label{thm:submotapprox}
Let $K$ be $L_{\omega,\omega}$ based on
a~signature~$\Omega$, let $K_\Di$ be $K$ enriched with
the modal operator for the extension relation on
models~$\sqsubseteq$, and let $\clC$ be an elementary (in $K$) class of models. For every $\kappa\geq\omega+|\Omega|$,
the following statements hold:
\begin{enumerate}
\item
Let $X$ be a~set of elements of $\stA\in\clC$,
and let $\lambda$ be a~cardinal such that
$|X|+\kappa\leq\lambda\leq|\stA|$. Then $\stA$~has
a~submodel~$\stB$ of cardinality $\lambda$ such that
$\Th^{K_\Di}(\stA)=\Th^{K_\Di}(\stB)$ and
$\stB$ contains~$X$.
\item
$
\{\Th^{K_\Di}(\stA):\stA\in\clC\}=
\{\Th^{K_\Di}(\stA):\stA\in\clC_{\leq\kappa}\}.
$
\item
The modal algebras of $K$ on $(\clC,\submod)$
and on $(\clC_{\leq\kappa},\submod)$ coincide.
\item
$
\MTh^K(\clC,\submod)=
\MTh^K(\clC_{\leq \kappa},\submod).
$
\end{enumerate}
\end{theorem}

\hide{We say that an equivalence $\equiv$ on $\clC$
is a {\em bisimulation}
{\em
w.r.t.~$\clR$ on $\clC$} if
for all $\stA,\stA',\stB\in\clC$ we have
\begin{equation}\label{eq:equive-bisim}
\stA\equiv\stA'\,\&\,\stA\clR\stB\;\Rightarrow\; \EE \stB'\in\clC \, (\stB'\equiv\stB\;\&\; \stA'\clR\stB').
\end{equation}
}

First, we provide a~general observation about any $\clC$
and $\clR$ satisfying the condition~(\ref{eq:isom-bisim}),
and arbitrary $K$.
Let $\equiv$ be the $K$-equivalence on $\clC$, and
let $\equiv_\Di$ be the $K_\Di$-equivalence on~$\clC$.
Hence, $\stA\equiv\stB$ means that
$\Th^K(\stA)=\Th^K(\stB)$, and $\stA\equiv_\Di\stB$
that $\Th^{K_\Di}(\stA)=\Th^{K_\Di}(\stB)$.
We recall that $\Th^{K_\Di}(\stA)$ depends not
only on $\stA$ but also on $\clC$ and~$\clR$.

\begin{proposition}\label{prop:equive-str}
Assume that $\clC$ and $\clR$
satisfy~ (\ref{eq:isom-bisim}). Let $\equiv$~be
a~bisimulation w.r.t.~$\clR$ on $\clC$, i.e.,
for all $\stA,\stA',\stB\in\clC$ we have
\begin{equation}\label{eq:equive-bisim}
\stA\equiv\stA'\;\&\;\stA\:\clR\:\stB
\;\Rightarrow\;
\EE\stB'\in\clC\;
(\stB'\equiv\stB\;\&\;\stA'\,\clR\,\stB').
\end{equation}
Then $\equiv$ and $\equiv_\Di$ coincide on~$\clC$.
\end{proposition}

\extended{\marISH{Can it be generalized for
the refinement?}}

\begin{proof}
By the standard bisimulation argument:
an easy induction on the construction
of $\mpsi(\mpv_1,\ldots,\mpv_n)$ shows that,
whenever $\stA\equiv\stA'$ and
$\vf_1,\ldots,\vf_n\in K_s$ then we have\,
$
\stA\mo\mpsi(\vf_1,\ldots,\vf_n)
\,\Iff\,
\stA'\mo\mpsi(\vf_1,\ldots,\vf_n).
$
\end{proof}

\begin{proposition}\label{prop:equive-bisim}
Let $K=L_{\omega,\omega}$. If $\clC$ is an
elementary class, then $\equiv$ is a~bisimulation
w.r.t.~$\sqsubseteq$ on~$\clC$.
\end{proposition}

\begin{proof}
Pick $\stA,\stA',\stB$ in~$\clC$.
If $\stA\equiv\stA'$, then there exists an
ultrafilter~$D$ such that the ultrapowers
$\stA_D$ and $\stA'_D$ are isomorphic,
due to the Keisler--Shelah isomorphism theorem
(see, e.g., \cite[Theorem 6.1.15]{chang1990model}).
If $\stA\submod\stB$, then $\stA_D$~is embeddable
in~$\stB_D$, the ultrapower of~$\stB$. Since
$\stA'$~is embeddable in~$\stA'_D$, we obtain that
$\stA'$ is embeddable in~$\stB_D$, and thus
$\stA'\submod\stB'$ for some $\stB'\isom\stB_D$.
But then we have $\stB'\in\clC$ and $\stB'\equiv\stB$,
which completes the proof.
\end{proof}

\begin{proof}[Proof of Theorem \ref{thm:submotapprox}]
By Propositions \ref{prop:equive-str}
and~\ref{prop:equive-bisim}, $\equiv$ and $\equiv_\Di$
coincide on~$\clC$, i.e., for all $\stA',\stA$ in~$\clC$
we have:
\begin{equation}\label{eq:cool1}
\text{$\stA'$ and $\stA$ are $K$-equivalent}
\;\Iff\;
\textrm{$\stA'$ and $\stA$ are $K_\Di$-equivalent}.
\end{equation}
Now the first statement of the theorem follows from the
downward L\"owenheim--Skolem theorem for
the first-order case:
$\stA$ has an elementary submodel~$\stB$ of
cardinality~$\lambda$ such that
$\stB$ contains~$X$
(see, e.g., \cite[Theorem 3.1.6]{chang1990model});
hence $\stA\equiv_\Di\stB$ by~(\ref{eq:cool1}).

Let $(\Ths,\submod_\Ths,\clA)$ and
$(\Ths^\kpp,\submod^\kpp_\Ths,\clA^\kpp)$ be
 the minimal frames of theories
 for $\clC$, $\submod$, $K_\Di$ and
  for $\clC_{\leq \kappa}$, $\submod$, $K_\Di$, respectively.
Let us show that these two structures are equal.

We have $\Ths=\Ths^\kpp$, the second statement of the theorem, as an immediate corollary of the first one.
\hide{\marISH{Obsolete}
Suppose that $T\in\Ths$, that is, we have $T=\Th^{K_\Di}(\stA)$ for some $\stA\in\clC$.
We have $\stA'\equiv \stA$ for some countable $\stA'\in\clC$. By (\ref{eq:cool1}),
$\stA'$ and $\stA$ are $K_\Di$-equivalent. Thus $T\in\Ths^\kpp$.
It follows that $\Ths=\Ths^\kpp$.}

Suppose that $T_1\submod_\Ths T_2$, that is,
$T_1=\Th^{K_\Di}(\stA)$ and $T_2=\Th^{K_\Di}(\stB)$
for some $\stA,\stB\in\clC$ with $\stA\submod\stB$.
Then $\stA$ has a~submodel~$\stA'$ of
cardinality~$\leq\kappa$ with $\stA'\equiv_\Di\stA$
(indeed, if the cardinality of $\stA$ is less than~$\kappa$
then we put $\stA'=\stA$, otherwise we use the first
statement of the theorem). Likewise, $\stB$ has
a~submodel~$\stB'$ of cardinality $\leq\kappa$ such that
$\stB'\equiv_\Di\stB$ and $\stB'$ contains the universe
of~$\stA'$.
Clearly, $\stA'$ is submodel of $\stB'$.
It follows that $T_1\submod^\kpp_\Ths T_2$.
Hence $\submod^\kpp$ includes $\submod$.
The converse inclusion is obvious, so
$\submod^\kpp$ equals~$\submod$.

Let $\vf\in K_\Di$. Since $\Ths=\Ths^\kpp$,
it is immediate that
$\{T\in\Ths:\vf\in T\}=\{T\in\Ths^\kpp:\vf\in T\}$.
Hence, $\clA^\kpp=\clA$.

Thus
$
(\Ths,\submod_\Ths,\clA)=
(\Ths^\kpp,\submod^\kpp_\Ths,\clA^\kpp)
$.
Now the third and consequently, the forth, statements
are immediate from Definition \ref{def:external-sat}
and Theorem~\ref{thm:minmax-ths}.
\end{proof}

This result is based on the interplay between the downward L\"owenheim--Skolem property of first-order language and the fact
that $\equiv$ is a bisimulation w.r.t. $\sqsubseteq$.
The property of $\equiv$ being a bisimulation w.r.t.~a~given~$\mathcal R$
seems interesting enough per se. 
Contrary to the case of the relation $\submod$ (Proposition \ref{prop:equive-bisim}), we have:
\extended{\marISH{
\newISH{Does it follow that
$\MTh^K(\clC,\submod)$ is Kripke complete?
}
Denis: Seems yes for the robust case --- we have image-set now.
}}

\begin{proposition}
Let $\Omega$~contain a~predicate symbol of
arity~$\ge2$. Then
the (usual) elementary equivalence
is not a bisimulation
w.r.t.~$\sqsupseteq$ on the class of models
of~$\Omega$.
\end{proposition}

\begin{proof}
Assume w.l.g.~$\Omega$ contains a~single binary
predicate symbol~${\le}$, and let $\equiv$~denote
$\equiv_{L_{\omega,\omega}}$.
Let $\mathfrak A,\mathfrak A',\mathfrak B$ be
the linearly ordered sets $\mathbb Q\cdot\mathbb Z$,
$\mathbb Z$, $\mathbb Q$, respectively (where
$\,\cdot\,$~denotes the lexicographical multiplication).
We have $\mathfrak A\equiv\mathfrak A'$
(this fact
can be established by using an Eurenfeucht--Fra\"iss\'e
game, see, e.g.,
\cite[Proposition 2.4.10]{marker2002modelIntro})
and  $\stB$ is embeddable in $\stA$.
However, every $\mathfrak B'$ satisfying
$\mathfrak B'\equiv\mathfrak B$ is a~dense
linearly ordered set without end-points.
Clearly, no such~$\mathfrak B'$ can at the same
time
be embeddable in $\mathfrak A'$.
\end{proof}

Let us say that $\clR$~is {\em image-closed
under ultraproducts} iff for every $\mathfrak A$,
$(\mathfrak B_i)_{i\in I}$, and ultrafilter~$D$
over~$I$, if $\mathfrak A\,\clR\,\mathfrak B_i$
for all $i\in I$ then
$\mathfrak A\,\clR\prod_D\mathfrak B_i$
(e.g., $\mathfrak A\,\mathcal R\,\mathfrak B$
may mean  that, up to isomorphism, $\mathfrak B$~is
an extension~of~$\mathfrak A$).

\begin{proposition}
Let $\clC$ and $\clR$ satisfy~$(\ref{eq:isom-bisim})$,
let $\clR$~be image-closed under ultraproducts on~$\clC$,
and let $K=L_{\omega,\omega}$.
If $\equiv$ and $\equiv_\Di$ coincide on~$\clC$, then
$\equiv$ is a~bisimulation w.r.t.~$\clR$ on~$\clC$.
\end{proposition}

\begin{proof}
Let $\mathfrak A,\mathfrak A',\mathfrak B$ in~$\clC$
be such that $\mathfrak A\equiv\mathfrak A'$ and
$\mathfrak A\,\clR\,\mathfrak B$.
Observe that for all $\varphi\in K_s$ we have:
\begin{align*}
\mathfrak B\vDash\varphi
&\;\Rightarrow\;
\mathfrak A\vDash\Di\varphi
\\
&\;\Rightarrow\;
\mathfrak A'\vDash\Di\varphi
\;\;\;(\text{since }\mathfrak A\equiv_\Di\mathfrak A'\,)
\\
&\;\Rightarrow\;
\mathfrak A'\,\clR\,\mathfrak B_\varphi
\;\&\;
\mathfrak B_\varphi\vDash\varphi
\text{ for some }\mathfrak B_\vf\in\clC.
\end{align*}
Let $T=Th^K(\mathfrak B)$.
For any $\varGamma\in\mathcal P_{\omega}(T)$
choose a~model~$\mathfrak B_{\varGamma}$ in~$\clC$
as in the observation above, i.e., such that
$\mathfrak A'\,\clR\,\mathfrak B_{\varGamma}$ and
$\mathfrak B_{\varGamma}\vDash\bigwedge\varGamma$.
Pick any ultrafilter~$D$ extending the centered
family of sets
$
S_\varphi=
\{\varGamma\in\mathcal P_{\omega}(T):
\varphi\in\varGamma\}
$
for all $\varphi\in T$ (i.e., a~{\em fine}
ultrafilter over $\mathcal P_{\omega}(T)$).
Then the ultraproduct
$\mathfrak B'=\prod_D\mathfrak B_\varGamma$ satisfies
all $\varphi\in T$, thus $\mathfrak B'\vDash T$.
Moreover, since $\clR$~is image-closed
under ultraproducts, we have
$\mathfrak A'\,\clR\,\mathfrak B'$.
This completes the proof.
\end{proof}

\begin{remark}
It suffices to assume that $\clR$~is image-closed
under ultraproducts by ultrafilters over~$|K|$ only. The proposition remains true for $K=L_{\kappa,\lambda}$
with any strongly compact $\kappa$.
\end{remark}

\hide{
\marISH{Something wrong with it, see  reviewer's comment 21}
\begin{example}
Let $\Omega$~consist of unary predicate symbols
and possibly some constant symbols. It can be shown
that the language $L_{\omega,\omega}$
expresses the $\sqsupseteq$-satisfiability on $\Omega$-models. Therefore, the (usual)
elementary equivalence is a bisimulation
w.r.t.~$\sqsupseteq$ on any elementary class~$\clC$
of models of~$\Omega$ (e.g., of all these models).
\end{example}
}

\section{Robustness and Kripke completeness}\label{sec:rob}

The logics of submodels calculated in Section \ref{sec:sub} do not depend on
choosing particular language (Theorem \ref{thm:sqsup-completeness}).
However, in general,
modal theories depend on $L$.
\begin{example}\label{ex:non-rob1}
Let $L$ be the first-order language $L_{\omega,\omega}$, $T$ the theory of dense linear orders without end-points.
Trivially,  the $\supmod$-satisfiability is expressible on $\clC=\Mod(T)$ in $L$: put $f(\vf)=\vf$.
Then $\MTh^L(\clC,\supmod)$ contains the ``trivial'' formula $\mathsf{p} \leftrightarrow \Di \mathsf{p}$.
Obviously, this formula is falsified if $L$ is second-order.
\end{example}

Henceforth we assume that $\clC$ and $\clR$
satisfy (\ref{eq:isom-bisim}).

\begin{proposition}\label{prop:language-monot}
Let $L$ and $K$ be two languages.
Then $L\subseteq K$ implies
$
\mathrm{MTh}^{L}(\mathcal C,\mathcal R)
\supseteq
\mathrm{MTh}^{K}(\mathcal C,\mathcal R).
$
\end{proposition}

\begin{proof}
Follows from Definition \ref{def:external-sat},
since on $(\clC,\clR)$,  the modal algebra of~$L$
is a subalgebra of the algebra of~$K$.
\end{proof}

One can think that
$L$~describes the properties of $(\mathcal C,\mathcal R)$
in a~robust
way whenever the modal theory does not change under strengthening the language.
Thus,
we shall say that $\MTh^L(\mathcal C,\mathcal R)$
is {\em robust\/} iff for every language $K\supseteq L$
we have
$$
\mathrm{MTh}^{K}(\mathcal C,\mathcal R)=
\mathrm{MTh}^{L}(\mathcal C,\mathcal R).
$$
Intuitively, the robust theory can be considered as
a~``true'' modal logic of the model-theoretic relation
$\mathcal R$ on~$\mathcal C$.

\hide{
\marISH{Obsolete:}
It follows that Theorem~\ref{thm:sqsup-completeness} describes robust theories of $\supmod$
(by Proposition \ref{prop:language-monot} and the soundness  Theorem
\ref{thm:sqsup-soundness}).\marISH{Plural is confusing here.}
}

\smallskip
By Definition \ref{def:external-sat} and Theorem \ref{thm:maintoolnew},
modal theories are logics of general frames.  The following construction shows that,
under certain assumptions, robust theories are logics of  Kripke frames (i.e., they are {\em Kripke complete}).

A~relation~$\clS$ is said to be {\em image-set}
iff for every~$\stA$ the image
$\clS(\stA)=\{\stB:\stA\,\clS\,\stB\}$
is a~set; $\clS$~is {\em image-set on}~$\clC$ iff
its restriction to~$\clC$ is image-set.
Recall that $\mathfrak A\isom\mathfrak B$ means that
$\mathfrak A$ and~$\mathfrak B$ are isomorphic, and
that $\,[\stA]_\isom\clR_\isom[\stB]_\isom$ iff
$
\exists\,\mathfrak A'\isom\mathfrak A\,
\exists\,\mathfrak B'\isom\mathfrak B\;
(\mathfrak A'\,\mathcal R\,\mathfrak B').
$
Let $\mathcal C_\isom$ abbreviate $\mathcal C/{\isom}$.  Also, recall  that $\clR^*(\stA)$ denotes
$\bigcup_{n<\omega}\clR^n(\stA)$, the least
$\clR$-closed $\clD\subseteq\clC$ containing~$\stA$.

\begin{proposition}\label{prop:sup-isom}
If $\clR_{\isom}$ is image-set on~$\clC_\isom$,
then
\begin{center}
$\MTh^L(\clC,\clR) \supseteq
\,\bigcap\limits_{\mathfrak A\in\clC}
\MLog(\clR^*(\mathfrak A)_{\isom},\clR_\isom).$
\end{center}
\end{proposition}

\begin{proof}
Let $\stA\in\clC$. The inclusion
$
\MTh^{L}(\clR^*(\stA),\clR)\supseteq
\MLog(\clR^*(\mathfrak A)_{\isom},\clR_\isom)
$
follows from~(\ref{eq:quot})
at the end of Section~\ref{sec:defs}.
Now we apply Corollary~\ref{cor:gener}.
\end{proof}

\hide{$\clC$ and $\clR$ satisfy (\ref{eq:isom-bisim})}
\begin{theorem}\label{thm: robust logicGOOD}
If
$\clR_{\isom}$
is image-set on~$\clC_\isom$, then
\begin{center}
$\MTh^L(\clC,\clR)$ is robust \;iff\;
$
\MTh^L(\clC,\clR)=
\bigcap\limits_{\mathfrak A\in\clC}
\MLog(\clR^*(\mathfrak A)_{\isom},\clR_\isom).
$
\end{center}
\end{theorem}

\begin{proof}
$(\Rightarrow)$
Put $\Lambda=\MTh^L(\clC,\clR)$. In view of
Proposition~\ref{prop:sup-isom}, we only have to
show that $\Lambda$~is valid in the Kripke frame
$(\clR^*(\mathfrak A)_{\isom},\clR_\isom)$
for every $\stA\in\clC$.

Fix $\stA\in\clC$ and put $Y=\clR^*(\stA)$.
For a~language~$M$,
let $({\Ths^M}\!,\clR_{\Ths^M},{\clA^M})$ be
the minimal frame of its theories in the class~$Y$.

Since $\clR_{\isom}$ is image-set on~$\clC_\isom$,
the quotient $Y_{\isom}$ is a~set.
Hence we can choose a~language~$K$ stronger than~$L$
and such that the $K$-equivalence and the isomorphism
relation coincide on~$Y$. Observe that for every $M$
stronger than~$K$, the Kripke frames
$({\Ths^M}\!,\clR_{\Ths^M})$,
$({\Ths^K}\!,\clR_{\Ths^K})$, and
$(Y_{\isom},\clR_\isom)$ are isomorphic.
We choose such~$M$ that
\begin{eqnarray*}
&\AA\, T\in{\Ths^K}\:
\EE\,\vf_T\in M_s\:
\AA\,\stA\in Y\;
\bigl(\stA\mo\vf_T\,\Iff\,\Th^K(\stA)=T\bigr),
&\!\text{and moreover}
\\
&\AA\,\clS\subseteq{\Ths^K}\:
\EE\,\vf_\clS\in M_s\:
\AA\,\stA\in Y\;
\bigl(\stA\mo\vf_\clS\,\Iff\,\Th^K(\stA)\in\clS\bigr).&
\end{eqnarray*}
It follows that ${\clA^M}$~is $\mathcal P(\Ths^M)$,
the powerset of~$\Ths^M$, and so
$(\Ths^M\!,\clR_{\Ths^M},{\clA^M})$ is
a~Kripke frame. By Theorem~\ref{thm:maintoolnew},
$\MTh^M(Y,\clR)$ is the logic of this frame. Therefore,
$\MTh^M(Y,\clR)$ is the logic of the Kripke frame
$(Y_{\isom},\clR_\isom)$. Since $\Lambda$~is robust,
we have $\Lambda = \MTh^M(\clC,\clR)$.
By Corollary~\ref{cor:gener},
$\Lambda \subseteq \MTh^M(Y,\clR)$.

$(\Leftarrow)$ Immediate from Propositions
\ref{prop:language-monot} and~\ref{prop:sup-isom}.
\end{proof}

\hide{
\begin{proof}
$(\Rightarrow)$
Put $\Lambda=\MTh^L(\clC,\clR)$. In view of
Proposition~\ref{prop:sup-isom}, we only have to
show that $\Lambda$~is valid in the Kripke frame
$(\clR^*(\mathfrak A)/{\isom},\clR_\isom)$
for every $\stA\in\clC$.

Fix $\stA\in\clC$ and put $Y=\clR^*(\stA)$.
Since $\clR_{\isom}$ is image-set on~$\clC_\isom$,
the quotient $Y/{\isom}$ is a~set. For a~language~$M$,
let $({\Ths(M)},\clR_{\Ths(M)},{\clA(M)})$ be
the minimal frame of its theories in the class~$Y$.
Hence we can choose a~language~$K$ stronger than~$L$
and such that the $K$-equivalence and the isomorphism
relation coincide on~$Y$. Observe that for any $M$
stronger than~$K$, the Kripke frames
$({\Ths(M)},\clR_{\Ths(M)})$,
$({\Ths(K)},\clR_{\Ths(K)})$, and
$(Y/{\isom},\clR_\isom)$ are isomorphic.
We choose such~$M$ that
\begin{eqnarray*}
&\AA\,T\in{\Ths(K)}\:
\EE\,\vf_T\in M_s\:
\AA\,\stA\in Y\;
\bigl(\stA\mo\vf_T\,\Iff\,\Th^K(\stA)=T\bigr),
&\!\text{and moreover}
\\
&\AA\,\clS\subseteq{\Ths(K)}\:
\EE\,\vf_\clS\in M_s\:
\AA\,\stA\in Y\;
\bigl(\stA\mo\vf_\clS\,\Iff\,\Th^K(\stA)\in\clS\bigr).&
\end{eqnarray*}
It follows that ${\clA(M)}$~is the powerset of~$\Ths(M)$,
and so $(\Ths(M),\clR_{\Ths(M)},{\clA(M)})$ is a~Kripke
frame. By Theorem \ref{thm:maintoolnew},
$\MTh^M(Y,\clR)$ is the logic of this frame. Hence,
$\MTh^M(Y,\clR)$ is the logic of the Kripke frame
$(Y/{\isom},\clR_\isom)$. Since $\Lambda$~is robust,
we have $\Lambda = \MTh^M(\clC,\clR)$.
By Corollary~\ref{cor:gener},
$\Lambda \subseteq \MTh^M(Y,\clR)$.

$(\Leftarrow)$ Immediate from Propositions
\ref{prop:language-monot} and~\ref{prop:sup-isom}.
\end{proof}
}

Hence, when $\clR_{\isom}$ is image-set on $\clC_\isom$, this theorem describes a unique modal logic, the {\em robust theory of $(\clC,\clR)$}.
\hide{; we
shall denote it by $\mathrm{MTh}(\mathcal C,\mathcal R)$  (without a~superscript).\marISH{Remove}}

\begin{remark}
More generally, if $\mathcal K$~is a~class of model-theoretic
languages, a~modal theory $\MTh^L(\mathcal C,\mathcal R)$
is \emph{robust in\/}~$\mathcal K$ iff it coincides with
$\MTh^K(\mathcal C,\mathcal R)$ for every language
$K\supseteq L$ in~$\mathcal K$. It is easy to see that there
exists at most one robust theory of $(\mathcal C,\mathcal R)$
whenever $\mathcal K$~is directed, and exactly one such
theory whenever $\mathcal K$~is countably directed, in
the sense that for every $\{K_n\in\mathcal K:n\in\omega\}$
there is $K\in\mathcal K$ which is stronger than any~$K_n$.
\end{remark}

Notice that even if $\clC$ is a set, Theorem \ref{thm: robust logicGOOD}
does not mean that
the robust theory of
$(\clC,\clR)$ is the logic
of the Kripke frame $(\clC,\clR)$.

\begin{example}
Let $\stA$ be the algebra
$(2^{<\omega},\inf,{\mathrm l},{\mathrm r})$
with the binary operation $\inf(x,y)=x\cap y$
and the unary operations
${\mathrm l}(x)=x^\conc 0$,
${\mathrm r}(x)=x^\conc 1$,
where ${}^{\conc}$~denotes the concatenation.
It is easy to see that $(\Sub(\stA),\supmod)$,
the structure of submodels of~$\stA$, is isomorphic
to the binary tree $\frT_2=(2^{<\omega},\subseteq)$.
The modal logic of the Kripke frame $\frT_2$
is known to be $\lS{4}$ (see~\cite{GoldMink1980}).
However, for every~$L$, the modal theory $\MTh^L(\Sub(\stA),\supmod)$
is the trivial logic given by the axiom
$\mathsf{p}\leftrightarrow\Di\mathsf{p}$ because
every submodel of $\stA$ is isomorphic to~$\stA$.
\end{example}

\paragraph{More completeness results}\label{par:morecompl}
We apply Theorem \ref{thm: robust logicGOOD}
to describe robust theories of the quotient and the submodel relations on certain natural classes.

The following fact most probably was known since 1970s. 
\begin{proposition}[V.\,B.~Shehtman,
private communication]\label{prop:MedvedevTop}
The logic of the Kripke frame
$(\mathcal P(\omega),\subseteq)$ is $\mathrm{S4.2.1}$.
\end{proposition}

\begin{proof}[Proof (sketch)]
It is not difficult to construct a family
$\clV\subset\mathcal P(\omega)$ such that
$\clV$ is downward-closed (i.e.,
$U_1\subseteq U_2\in\clV$ implies $U_1\in \clV$)
and there exists a~p-morphism of $(\clV,\subseteq)$
onto the binary tree $\frT_2=(2^{<\omega},\subseteq)$.
(E.g., let $f$ be a~bijection $\omega\to 2^{<\omega}$,
and let $\clV$ consist of $U\subset \omega$ such that $f(U)$
is a~finite chain in~$\frT_2$; the required p-morphism
takes $U$ to the greatest element of this chain.)
This p-morphism obviously  extends to the p-morphism
of $(\clP(\omega),\subseteq)$ onto~$\frT_2'$,  the ordered sum of $\frT_2$ and a~reflexive
singleton. Since $\MLog(\frT_2')=\mathrm{S4.2.1}$
(see~\cite{ShehtMink83}), it follows that
${\MLog(\mathcal P(\omega),\subseteq)}$ is included
in $\mathrm{S4.2.1}$. The converse inclusion is clear.
\end{proof}

\begin{theorem}\label{thm:submRob}
Let $\Omega$ contain a functional
symbol and a constant symbol, and let
$\mathcal C$~be the class of all models of~$\Omega$.
Then the robust theory of $(\mathcal C,\sqsupseteq)$ is $\mathrm{S4.2.1}$.
\end{theorem}

\begin{proof}
Soundness is straightforward.
Let us check the converse inclusion.

\hide{
For any $n\in\omega{\setminus}\{0\}$, let
$\mathfrak A_n$~be a~one-generated unar, i.e.,
a~model of a~signature~$\Omega_0$ consisting of
a~single unary functional symbol~$F$, such that
$n$~is the least~$m>0$ for which $\mathfrak A_n$~satisfies
the identity $F^m(x)=x$; thus $\mathfrak A_n$~can be
considered as a~cycle of length~$n$. As easy to see,
each $\mathfrak A_n$~has no proper submodels.
}

Let a~signature~$\Omega_0$ consist of a~single unary
functional symbol~$F$; models of this signature are
called {\em unars}. For any $n\in\omega{\setminus}\{0\}$,
let $\mathfrak A_n$~be  a~cycle of length~$n$, i.e.,
a~one-generated unar  of
cardinality~$n$  that satisfies
$F^n(x)=x$. As easy to see,
each $\mathfrak A_n$~has no proper submodels.

Let also $\mathfrak A$~be the disjoint sum of
all these~$\mathfrak A_n$. Clearly, $\mathfrak A$~is
a~countably generated unar, and any its submodel is
given by a~nonempty set $S\subseteq\omega$, and the
map taking $S$ to the disjoint sum of $\mathfrak A_n$
for all $n\in S$ is an isomorphism between
$(\mathcal P(\omega){\setminus}\{\emptyset\},\supseteq)$
and the frame $(Sub(\mathfrak A),\sqsupseteq)$ of
submodels of~$\stA$.
Moreover, all submodels of~$\mathfrak A$ are
pairwise non-isomorphic, and hence, the
Kripke frame
$(Sub(\mathfrak A)_{\isom},\sqsupseteq_{\isom})$
is also isomorphic to
$(\mathcal P(\omega){\setminus}\{\emptyset\},\supseteq)$.

Without loss of generality we may assume that
$F$ is the only functional symbol in~$\Omega$
(we can always mimic a unary functional symbol by
any functional symbol of arity $\geq 1$).
Expanding $\Omega_0$ to~$\Omega$, let $\mathfrak A'$
be the model of~$\Omega$ obtained from $\mathfrak A$
by adding a~single extra point $a$ which interprets all
constant symbols and assuming $F(a)=a$.
\marISH{DC}%
Then the Kripke frame
$(Sub(\mathfrak A')_{\isom},\sqsupseteq_{\isom})$
is isomorphic to $(\mathcal P(\omega),\supseteq)$.

By Theorem \ref{thm: robust logicGOOD}, the robust
theory of $(\mathcal C,\sqsupseteq)$ is included in
the logic $\MLog(\mathcal P(\omega),\supseteq)$, which is
$\mathrm{S4.2.1}$ by Proposition~\ref{prop:MedvedevTop}.
\end{proof}

\begin{remark}
It follows from the above proof that for
the class~$\mathcal C$ of unars without
constant symbols, the robust theory of
$(\mathcal C,\sqsupseteq)$ is included in
the intersection of the logics of {\em Skvortsov frames} $(\clP(\kappa){\setminus}\{\emptyset\},\supseteq)$ for all $\kappa$.
However, it is unclear how the ``soundness'' part of
the proof can be obtained. This question is connected
to a long-standing open problem about properties of
the modal Medvedev logic, see, e.g.,~\cite{ShehtMedv1990}.
\end{remark}

\begin{remark}\label{rem:compareTheorems}
The proof of Theorem \ref{thm:submRob} is much simpler
than the proof of Theorem \ref{thm:sqsup-completeness},
and covers more signatures in the case with constants.
However, it is unclear whether the completeness result
holds for second- or first-order language. In other words,
we do not know whether the resulting theories are robust
in the case of second- or first-order language.
\end{remark}

Let $\mathfrak B\le\mathfrak A$ mean that
$\mathfrak B$~is a~quotient of~$\mathfrak A$
(for the definition of quotients of arbitrary models,
see~\cite[Section 2.4]{Malcev73}), and
$
Quot(\mathfrak A)=\{\mathfrak B:
\mathfrak B\le\mathfrak A\}.
$
Quotients are, up to isomorphism, images under
strong homomorphisms, hence the logic of quotients is the
same that the logic of strong homomorphic images.

\begin{theorem}\label{thm:quot}
Let $\Omega$ have a~functional symbol of arity~$\ge1$,
and $\mathcal C$ the class of all models of~$\Omega$.
Then the robust theory of $(\mathcal C,\geq)$
is~$\mathrm{S4.2.1}$.
\end{theorem}

\begin{proof}\marISH{New. Double check}
Soundness is straightforward. In particular, every
model~$\stA$ has a~unique single-point quotient~$\stB$,
and for  any valuation $\vl$ in $L$ we have:
${\stA\mo\vlf{\Box\Di \mathsf{p}}}$ iff
$\stB\mo\vlf{\Di \mathsf{p}}$ iff
$\stB\mo\vlf{\Box \mathsf{p}}$ iff
$\stA \mo \vlf{\Di \Box \mathsf{p}}.$

\marISH{Formally, $\geq$ is not transitive, and I skipped one trivial step here. Denis, are you OK with the above statement?}%
By Theorem \ref{thm: robust logicGOOD},
the robust theory of $(\mathcal C,\ge)$ is included in
the logic $(Quot(\mathfrak A)_{\isom},\ge_{\isom})$ for
any model~$\stA$ of~$\Omega$.
We shall construct $\mathfrak A$  such that
$\MLog(Quot(\mathfrak A)_\isom,\ge_\isom)\subseteq \mathrm{S4.2.1}$.
It suffices to handle the case when $\Omega$~consists
of a~single unary functional symbol $F$.
As in
Theorem~\ref{thm:submRob}, let $\mathfrak A_n$~be  a~one-generated unar forming the
cycle of length~$n$.
\hide{
As in
Theorem~\ref{thm:submRob}, let $\mathfrak A_n$~be
a~one-generated unar satisfying $F^n(x)=x$ with the least
such~$n>0$ (forming a~cycle of length~$n$).
}%
 It is easy to see that
all quotients of $\mathfrak A_n$ are (up to isomorphism)
exactly $\mathfrak A_m$ for $n$ divisible by $m$; in
particular, if $p$~is prime, $Quot(\mathfrak A_p)$ is $\{\mathfrak A_1,\mathfrak A_p\}$.

Let $\mathfrak A$~be the disjoint sum of the
unars~$\mathfrak A_p$ for $p=1$ or $p$~prime.
Let $P$ denote the set of primes.
Up to isomorphism, every quotient $\stB$ of $\stA$
is the disjoint sum of unars $\{\stA_p:p\in S\}$ and
also $n$ copies of $\stA_1$ for some $S\subseteq P$
and $n\leq 1+|P{\setminus}S|$;
we put $h(\stB)=h([\stB]_\isom])=S$ and $g(\stB)=n$.
We claim that $h$ is a p-morphism of
$(Quot(\mathfrak A)_\isom,\ge_\isom)$ onto
$(\clP(P),\supseteq)$. Surjectivity and monotonicity
are straightforward. Assume that $h([\stB]_\isom)=S$ and
pick any $S'\subseteq S$.
Let $\stB'$ be the disjoint sum of unars
$\{\stA_p:p\in S'\}$ and $g(\stB)$ copies of~$\stA_1$.
Then $\stB'$ is isomorphic to a quotient of $\stB$, and $h([\stB']_\isom)=S'$, as required.

It follows that
$\MLog(Quot(\mathfrak A)_\isom,\ge_\isom)$ is
included in $\MLog(\clP(P),\supseteq)$.
By Proposition~\ref{prop:MedvedevTop}, the latter
logic is $\lS{4.2.1}$, which completes the proof.
\end{proof}

\begin{remark}
This theorem remains true for (not necessarily strong)
homomorphic images as well.
That $\lS{4.2.1}$ includes the modal theory of
homomorphisms follows from the proof above
(the signature of $\stA$ does not have predicate symbols).
And $\lS{4.2.1}$  is sound since
$\mathcal C$ has a single-point model that gives
a~top element in $\mathcal C/{\isom}$\,.
\end{remark}

\begin{remark}\label{ren:quot}\marISH{new}
Similarly to the case of the submodel relation,
the quotient relation is expressible in an appropriate
second-order language. Given $\Omega$ and
$\lambda>|\{F\in\Omega:F$~is a~functional symbol$\}|$,
for every $\varphi$ of $L^{2}_{\lambda,\omega}$
in~$\Omega$ we pick a~fresh binary predicate variable~$U$
and define $\varphi_U$ by induction on~$\varphi$:
if $t_0,\ldots,t_{n-1}$ are terms in~$\Omega$ and
$P(t_0,\ldots,t_{n-1})$ is an atomic formula, then
\begin{itemize}
\item[]
$P(t_0,\ldots,t_{n-1})_U$ is
$\exists x_0\ldots\exists x_{n-1}
\bigl(\bigwedge_{i<n}U(x_i,t_i)
\wedge P(x_0,\ldots,x_{n-1})\bigr)
$
\end{itemize}
where $x_i$~are fresh first-order variables;
if $F$~is a~functional variable, then
\begin{itemize}
\item[]
$(\exists F\,\varphi)_U$ is
$
\exists F\,(\text{$U$~is a~congruence for~$F$}
\wedge\varphi_U);
$
\end{itemize}
we let the operation distribute w.r.t.~Boolean
connectives and quantifiers over first-order and
predicate variables (e.g., $(\exists P\,\varphi\bigr)_U$
is $\exists P\,\varphi_U$ for every predicate
variable~$P$). If $\chi(U)$~says that the interpretation
of~$U$ is a~congruence for interpretations of all
functional symbols in~$\Omega$, then the map
$\varphi\mapsto\exists U(\chi(U)\wedge\varphi_U)$
expresses the $\geq$-satisfiability on $\Omega$-models.

However, in second-order language, we do not know
whether the modal theory of quotients
is $\lS{4.2.1}$.
\end{remark}
\extended{

The same can be done for $\mathcal L^{2}_{\lambda,\nu}$.
However, we do not know whether
$\MTh^L(\mathcal C,\geq)$ equals S.4.2.1 if $L$~is
$L^{2}_{\lambda,\omega}$ (or $L_{\lambda,\omega}$,
or $\mathcal L^{2}_{\lambda,\nu}$).
}

So far we have calculated modal theories only for
classes consisting of all models of a~given signature.
In the following examples, we describe robust modal
theories of classes consisting of models of a~non-trivial
theory; we only outline ideas and postpone complete
proofs for a~further paper.

\begin{example}
Let $\mathcal D$~be the class of dense linearly
ordered sets without end-points. It was observed
in Example~\ref{ex:non-rob1} that the modal theory
$\MTh^{L_{\omega,\omega}}(\mathcal D,\sqsupseteq)$
is trivial: this is the logic of a~reflexive singleton
axiomatized by the formula
$\mathsf{p}\leftrightarrow\Diamond\mathsf{p}$.
On the other hand, the robust theory
of $(\mathcal D,\sqsupseteq)$ is $\mathrm{S4.2.1}$.

By the classical Cantor theorem, every two
countable orders in $\mathcal D$ are isomorphic
(see, e.g., \cite[Theorem~2.8]{rosenstein1982linear}). Therefore,
$(\mathcal D_{\isom},\sqsupseteq_{\isom})$ has a~top
element (the order type of rationals), so its logic
includes~$\mathrm{S4.2.1}$ (in spite of the fact that
we do not have constant symbols in the signature).
To prove the converse inclusion, in view of
Corollary~\ref{cor:gener}, it suffices to show that
S4.2.1 includes the logic of a~generated subframe of
the frame $(\mathcal D_{\isom},\sqsupseteq_{\isom})$.

Given $S\subseteq\mathbb R$, let
$D_S=\{[X]_{\isom}:X\subseteq S$~is dense in~$\mathbb R\}$.
(Note that $[X]_{\isom}$ is the order type of~$X$.)
First we observe that, whenever $S\subseteq\mathbb R$
is dense, then $(D_S,\sqsupseteq_{\isom})$ forms an
upper cone of $(\mathcal D_{\isom},\sqsupseteq_{\isom})$.
Hence it suffices to find $S$ such that the logic of
$(D_S,\sqsupseteq_{\isom})$ coincides with S4.2.1.
For this, we use the following result by Sierpi\'nski:
there are two disjoint sets $E,F$ of reals both dense
in~$\mathbb R$, having cardinality $|E|=|F|=2^{\aleph_0}$,
and such that $f(E)\not\subseteq E\cup F$ for any
non-identity order embedding $f:\mathbb R\to\mathbb R$
(see, e.g., \cite[Chapter~9,~\S\,2]{rosenstein1982linear}).
Let $S=E\cup F$, and for any dense subset~$X$ of $S$, let
$\pi([X]_\isom)=A$ if $X=E\cup A$ for some $A\subseteq F$,
and $\pi([X]_\isom)=\emptyset$ otherwise. It can be
verified that $\pi$~is a~well-defined map and moreover,
a~p-morphism of $(D_S,\sqsupseteq_{\isom})$ onto
$(\mathcal P(F),\supseteq)$. It induces
a~p-morphism of $(D_S,\sqsupseteq_{\isom})$
onto $(\mathcal P(\omega),\supseteq)$.
Applying Proposition~\ref{prop:MedvedevTop},
we conclude that $\MLog(D_S,\sqsupseteq_{\isom})$
is $\mathrm{S4.2.1}$, as required.

Let $\mathcal L$ and $\mathcal O$ be the classes
of linearly ordered sets and partially ordered sets,
respectively. The robust theories of
$(\mathcal L,\sqsupseteq)$ and $(\mathcal O,\sqsupseteq)$
also coincide with $\mathrm{S4.2.1}$. To see this,
we note that there exists a~p-morphism of
$(\mathcal L,\sqsupseteq)$ onto
$(\mathcal D,\sqsupseteq)$ (condensing scattered segments, see, e.g., \cite[Chapter~4]{rosenstein1982linear}) and
that $(\mathcal L,\sqsupseteq)$ is an upper cone of
$(\mathcal O,\sqsupseteq)$. This induces the same relationships between
$(\mathcal D_\isom\sqsupseteq_\isom)$,
$(\mathcal L_\isom,\sqsupseteq_\isom)$, and
$(\mathcal O_\isom,\sqsupseteq_\isom)$.
\end{example}

\begin{example}
Let $\clC$ be the class of modal algebras.
Similarly to the proof
of Theorem~\ref{thm:submRob}, one can show that the robust theory
of $(\clC,\quot)$ is $\lS{4.2.1}$. Namely, let
$\frF$ be the disjoint sum of Kripke frames
$\frF_n=(n,n\times n)$, and
$\stA$ be the modal algebra of $\frF$. One can construct a p-morphism from
$(Quot(\stA)_\isom,\quot_\isom)$ onto
$(\clP(\omega),\supseteq)$. Thus, $\logicts{S4.2.1}$ is the robust theory of the quotient relation on the class $\clC$ of modal algebras.
We have the same axiomatization in the case when $\clC$ is the class of $\logicts{S4}$-algebras or $\logicts{S5}$-algebras,  because $\stA$ is an $\lS{5}$-algebra. (More generally, this holds if $\clC$ contains $\stA$ and is closed under quotients.)

The axiomatization in the case when $\clC$ is the  class
of Boolean algebras is an open question.
\end{example}

\begin{example}\label{ex:freegr}
Let $\mathcal C$~be the class of free groups.
 Let
$\mathbb F_\kappa$~be a~$\kappa$-generated free group.
All~$\mathbb F_\kappa$ with
$2\le\kappa\le\aleph_0$ are pairwise embeddable and
non-isomorphic, while on all other~$\mathbb F_\kappa$
their ordering by embedding coincides with their
ordering by rank (the least cardinality of generators).
Therefore, the Kripke frame
$(\mathcal C_\isom,\sqsupseteq_\isom)$ has the top
(corresponding to~$\mathbb F_1$),
a~countable cluster immediate below the top
(corresponding to $\mathbb F_\kappa$'s with
$2\le\kappa\le\aleph_0$), and a~structure
isomorphic to $(Ord,\ge)$
(corresponding to $\mathbb F_\kappa$'s with
$\kappa\ge\aleph_1$).
By Theorem~\ref{thm: robust logicGOOD}, it follows that
the robust theory of
$(\mathcal C,\sqsupseteq)$  is the modal logic of ordered sums
$(\alpha,\geq)+S+S_0$ for $\alpha\in Ord$,  a countable cluster $S$, and a singleton $S_0$.
By the standard filtration technique, this logic has the finite model property (more precisely, we may assume that
$\alpha$ and $S$ are finite) and decidable.

The class $\clC$ is closed under~$\sqsupseteq$ as any subgroup
of a~free group is free by the classical
Nielsen--Schreier theorem. Hence,
$\sqsupseteq$-satisfiability on~$\mathcal C$
is expressible in
$L^{2}_{\omega,\omega}$
by Proposition~\ref{prop:submod-quot-exprr}.
We do not know whether
the modal theory of $\sqsupseteq$ on~$\mathcal C$
in $L^{2}_{\omega,\omega}$ is robust.

We remark that in the first-order case the theory
is not robust. Let $\equiv$ be the usual elementary
equivalence. Observe that
$\mathbb F_1\isom\mathbb Z\not\equiv\mathbb F_2$
(obvious), and $\mathbb F_\kappa\equiv\mathbb F_2$
for all $\kappa\ge2$ (by recently proved famous
Tarski's conjecture; see, e.g., \cite[Chapter~9]{ElementaryGroups}).
It follows that $(\mathcal C_\equiv,\sqsupseteq_\equiv)$
is
isomorphic to the ordinal~$2$ (with the top
corresponding to~$\mathbb F_1$ and the bottom
to other free groups).
Moreover, the logic
$\MTh^{L_{\omega,\omega}}(\clC,\sqsupseteq)$ is
the logic of the ordinal~2. This follows from
Corollary~\ref{prop: finite frame}
in view of the following observation made by one
the reviewers on an earlier version of the paper:
the function $f:L_s\to L_s$ defined by letting
\[
f(\vf)=\left\{
\begin{array}{ll}
\top  & \text{ if }
\mathbb F_1\mo \vf,\\
\vf  & \text{ if }
\mathbb F_1\not\mo\vf\text{ but }\mathbb F_2\mo\vf,\\
\bot  & \text{ otherwise}
\end{array}
\right.
\]
expresses the $\sqsupseteq$-satisfiability on $\clC$.
This observation can be generalized as follows: if
the elementary equivalence $\equiv$ is a~bisimulation
w.r.t.~$\clR$ on $\clC$ and $\clC_\equiv$ is finite, then
the $\clR$-satisfiability on $\clC$ is first-order expressible.

\hide{
\marISH{DC}
Finally, notice that on the class of free groups
on $\kappa\ge2$ generators, the relations $\sqsupseteq$,
$\geq$, and $\succeq$ (of being an elementary submodel)
coincide (see \cite[Chapter~9]{ElementaryGroups}).
Thus this example can be modified for modal theories of free groups
endowed with $\geq$ and $\succeq$ as well. }
\end{example}

\hide{
The proof of Theorem \ref{thm:sqsup-completeness} is the least
technically elegant completeness proof in this work.\marISH{now ``least
technically elegant''  sounds a bit strange to me...}
However, this theorem is the only
result
describing the robust theories in
the case of second-, and moreover (according to
Corollary~\ref{cor:first-robust-sub}), first-order language.
\marISH{This repeats Remark
\ref{rem:compareTheorems}}

Let us note that the binary operation used
in the proof of Theorem~\ref{thm:sqsup-completeness}
is not associative. Thus, the modal axiomatizations
of $\supmod$ on monoids or (semi)groups in
first- or second-order languages are open
questions. Analogous
questions are open for the quotient relation,
as well as for other relations and theories.
\marISH{Sounds very strange...}
}

\smallskip

Even in the case of all models of a~given signature,
the axiomatization of the robust theory can be a~very
difficult problem. Consider the theories of the
relation~$\supmod$ on the class~$\clC$ of all models
of a~signature~$\Omega$ consisting of unary predicates
and perhaps some constants.

The easiest is the degenerated case when $\Omega$ has
no symbols other than constants. The theory of submodels
on this class is $\logicts{Grz.3}$, the
Grzegorczyk logic with the linearity axiom. As well-known,
this is the logic of conversely well-ordered sets, or
of the set~$\omega$ with the converse order, or else
of all finite linearly ordered sets.

Furthermore, these theories include the logic $\logicts{Grz}$ iff $\Omega$ contains finitely many predicates. We announce that in the case of
$n<\omega$~predicates, the robust modal theory coincides with the
modal logic of the direct product of $2^n$~copies of~$\omega$
with the converse order if the signature has constants, and
with the logic of this structure without the top element
otherwise. These logics have the finite model property, decrease as $n$~grows, and their
intersection coincides with $\MLog(\clP_{\omega}(\omega){\setminus}\{\emptyset\},\supseteq)$, the modal Medvedev logic.
Despite this clear semantical description, no  complete axiomatizations
for these logics are known.

\paragraph{Acknowledgements.}
We are grateful to the reviewers
for their multiple suggestions and questions on  earlier versions of the paper. We are also grateful to
Lev Beklemishev, Philip Kremer,  Fedor Pakhomov, Vladimir Shavrukov, Valentin Shehtman,  Albert Visser for their attention to our work,
useful comments, and discussions.

The work on this paper was supported by the Russian Science Foundation under grant 16-11-10252 and carried out at Steklov Mathematical Institute of Russian Academy of Sciences.

\bibliographystyle{sl}
\bibliography{modmod}



\hide{

\AuthorAdressEmail{Denis  \ts I. Saveliev}{Steklov Mathematical Institute of Russian Academy of Sciences \\
Institute for Information Transmission Problems of Russian Academy of Sciences
}
{d.i.saveliev@gmail.com}
\AdditionalAuthorAddressEmail{Ilya \ts B. Shapirovsky}
{Steklov Mathematical Institute of Russian Academy of Sciences\\
Institute for Information Transmission Problems of Russian Academy of Sciences
}
{ilya.shapirovsky@gmail.com}

}
\end{document}